\documentclass[12pt,twoside,a4paper]{amsart}
\usepackage{kvsetkeys}
\usepackage[utf8]{inputenc}
\usepackage{amssymb}
\usepackage{amsmath}
\usepackage{graphicx}
\usepackage{amsfonts}
\usepackage{mathrsfs}
\usepackage{mathptmx}
\usepackage{yhmath}
\usepackage{yfonts}
\usepackage{array}
\usepackage{makecell}
\usepackage[margin=1in]{geometry}
\usepackage{amsmath, amscd, accents}
\usepackage{stmaryrd}
\newcommand{\ol }{\overline }

\newcommand{\floor }[1]{\left\lfloor #1 \right\rfloor }
\newcommand{\ceiling }[1]{\left\lceil #1 \right\rceil }
\newcommand{\abs}[1]{\left\lvert #1 \right\rvert}

\newcommand{\ZZ }{\mathbb{Z}}

\newcommand{\CC }{\mathbb{C}}

\newcommand{\PP }{\mathbb{P}}
\newcommand{\OO }{\mathcal{O}}

\newcommand{\eps }{\varepsilon }
\newcommand{\crond }{\mathcal{C}}

\newcommand{\frond }{\mathcal{F}}

\newcommand{\mrond }{\mathcal{M}}
\newcommand{\grond }{\mathcal{G}}
\newcommand{\hrond }{\mathcal{H}}
\newcommand{\irond }{\mathcal{I}}
\newcommand{\lrond }{\mathcal{L}}

\newcommand{\srond }{\mathcal{S}}
\newcommand{\erond }{\mathcal{E}}

\newcommand{\lthen }{\Rightarrow }
\newcommand{\injto }{\hookrightarrow }

\newcommand{\Pic }{\textnormal{Pic}}

\newcommand{\Sym }{\textnormal{Sym}}
\newcommand{\ord }{\textnormal{ord}}
\newcommand{\gon }{\textnormal{gon}}
\newcommand{\Cliff }{\textnormal{Cliff}}
\newcommand{\rk }{\textnormal{rk}}

\PassOptionsToPackage{usenames,svgnames,dvipsnames,table}{xcolor}	
\usepackage{xcolor}
\usepackage[colorlinks=true]{hyperref}
\hypersetup{urlcolor=RubineRed,linkcolor=RoyalBlue,citecolor=ForestGreen}
\usepackage[nameinlink]{cleveref}
\usepackage{amsthm}
\usepackage{thmtools}
\usepackage[framemethod=TikZ]{mdframed}
\mdfdefinestyle{mdbluebox}{
roundcorner = 10pt,
linewidth = 1pt,
skipabove = 12pt,
innerbottommargin = 9pt,
skipbelow = 2pt,
linecolor = blue,
nobreak = true,
backgroundcolor = TealBlue!5,
}
\declaretheoremstyle[
headfont=\sffamily\bfseries\color{MidnightBlue},
mdframed={style=mdbluebox},
headpunct={\\[3pt]},
postheadspace={0pt}
]{thmbluebox}
\mdfdefinestyle{mdredbox}{
linewidth=0.5pt,
skipabove=12pt,
frametitleaboveskip=5pt,
frametitlebelowskip=0pt,
skipbelow=2pt,
frametitlefont=\bfseries,
innertopmargin=4pt,
innerbottommargin=8pt,
nobreak=true,
backgroundcolor=Salmon!5,
linecolor=RawSienna,
}
\declaretheoremstyle[
headfont=\bfseries\color{RawSienna},
mdframed={style=mdredbox},
headpunct={\\[3pt]},
postheadspace={0pt},
]{thmredbox}
\mdfdefinestyle{mdgreenbox}{
skipabove=8pt,
linewidth=2pt,
rightline=false,
leftline=true,
topline=false,
bottomline=false,
linecolor=ForestGreen,
backgroundcolor=ForestGreen!5,
}
\declaretheoremstyle[
headfont=\bfseries\sffamily\color{ForestGreen!70!black},
bodyfont=\normalfont,
spaceabove=2pt,
spacebelow=1pt,
mdframed={style=mdgreenbox},
headpunct={ --- },
]{thmgreenbox}
\mdfdefinestyle{mdblackbox}{
skipabove=8pt,
linewidth=3pt,
rightline=false,
leftline=true,
topline=false,
bottomline=false,
linecolor=black,			
backgroundcolor=RedViolet!5!gray!5,
}
\declaretheoremstyle[
headfont=\bfseries,
bodyfont=\normalfont\small,
spaceabove=0pt,
spacebelow=0pt,
mdframed={style=mdblackbox}
]{thmblackbox}
\theoremstyle{plain}
\declaretheorem[name=Theorem,numberwithin=section]{theorem}
\declaretheorem[name=Lemma,sibling=theorem]{lemma}
\declaretheorem[name=Proposition,sibling=theorem]{proposition}
\declaretheorem[name=Corollary,sibling=theorem]{corollary}
\declaretheorem[name=Theorem,numbered=no]{theorem*}
\declaretheorem[name=Lemma,numbered=no]{lemma*}
\declaretheorem[name=Proposition,numbered=no]{proposition*}
\declaretheorem[name=Corollary,numbered=no]{corollary*}

\declaretheorem[name=Definition,numbered=no]{definition*}
\declaretheorem[name=Conjecture,sibling=theorem]{conjecture}
\declaretheorem[name=Conjecture,numbered=no]{conjecture*}

\declaretheorem[name=Algorithm,numbered=no]{algorithm*}
\declaretheorem[name=Claim,sibling=theorem]{claim}
\declaretheorem[name=Claim,numbered=no]{claim*}

\declaretheorem[name=Case,numbered=no]{case*}

\declaretheorem[name=Example,numbered=no]{example*}

\declaretheorem[name=Examples,numbered=no]{examples*}
\theoremstyle{definition}
\declaretheorem[name=Remark,sibling=theorem]{remark}
\declaretheorem[name=Remark,numbered=no]{remark*}

\declaretheorem[name=Notation,numbered=no]{notation*}

\declaretheorem[name=Exercise,numbered=no]{exercise*}

\declaretheorem[name=Fact,numbered=no]{fact*}

\declaretheorem[name=Problem,numbered=no]{problem*}

\declaretheorem[name=Question,numbered=no]{question*}

\declaretheorem[name=Abuse of notation,numbered=no]{abuse*}
\declaretheorem[name=Acknowledgements,numbered=no]{acknowledgements*}
\usepackage{tikz}
\usepackage{quiver}
\usepackage{adjustbox}

\usepackage{fancyhdr}
\begin{document}
\title{{{The Stratified and the Strong Maximal Rank Conjecture\\ in $\PP^3$ and $\PP^4$}}}
\author{Vlad Robu}
\address{\vskip 1.5em Humboldt Universität zu Berlin, Unter den Linden 6, 10117 Berlin}
\email{vlad.nicolae.robu@hu-berlin.de}
\date{\today}
\maketitle

\begin{abstract}
    We prove that the Strong Maximal Rank Conjecture holds for quadrics in $\PP^3$ and we prove the existence of a component of the expected dimension in $\PP^4$, as well as in a wide range of parameters $(g,d)$ in $\PP^r$ with $r\ge 5$. We propose the Stratified Strong Maximal Rank Conjecture which also takes into account the rank $k$ of quadrics and prove it works in most of the cases when $k=3$ and $k=4$. We also prove a partial result that concerns the unrepresentability of the canonical bundle of a general curve as a sum of $3$ pencils. 
\end{abstract}

\section{Introduction}

Brill-Noether theory concerns the study of linear series on general curves. If $C$ is a general curve of genus $g$, the Brill-Noether theorem states that the schemes 
$$W^r_d(C)=\{L\in \Pic^d(C)\mid h^0(C,L)\ge r+1\}\textnormal{ and }$$
$$G^r_d(C)=\{(L,V)\mid L\in W^r_d(C),V\in G(r+1,H^0(C,L))\}$$
are of dimension $\rho(g,r,d)=g-(r+1)(r-d+g)$. In particular, if $\rho(g,r,d)<0$, the two schemes above are empty. When $\rho(g,r,d)\ge 0$, we know that $W^r_d(C)$ and $G^r_d(C)$ are reduced and the latter is also smooth. If furthermore $\rho(g,r,d)\ge 1$, then the two schemes are even irreducible. For more details concerning the basics of Brill-Noether theory, we refer to \cite{ACGH}. 

Consider a linear series $\ell=(L,V)\in G^r_d(C)$ and let $\varphi:C\stackrel{\abs{V}}{\longrightarrow}\PP^r$ be the induced map. The Maximal Rank Conjecture (MRC), now a theorem due to Eric Larson \cite{Larson17}, states that if $C$ is general in $\mrond_g$ and $\ell=(L,V)$ is also general in $G^r_d(C)$, then the multiplication maps of sections
$$\Sym ^nV\to H^0\left(C,L^n\right)$$
are of maximal rank. There are also earlier proofs in specific ranges, of which we mention the approach with tropical methods \cite{JP15} that settles the conjecture for $n=2$, as well as \cite{BE87},\cite{BF10} that are concerned with the nonspecial range and the surjective range respectively. In this generality setting, it is known that $\varphi:C\injto \PP^r$ is actually an embedding. Since we can calculate $\dim_\CC\Sym^n H^0(C,L)=\binom{r+n}{n}$ and the generality assumption on $C$ ensures $h^0(C,L^n)=nd+1-g$ via Riemann-Roch, MRC gives a complete description of the Hilbert function of the embedded curve. 

Having in mind important applications to the Kodaira dimension of $\mrond_g$ \cite{Farkas09} that eventually led to the recent proof \cite{FJP25} that $\ol{\mrond}_{22}$ and $\ol{\mrond}_{23}$ are both of general type, Aprodu and Farkas proposed in \cite{AF11} a refinement, called the \textbf{Strong Maximal Rank Conjecture (SMRC)}. Keeping the generality assumption only on the abstract genus $g$ curve $C$, the conjecture is now concerned with all linear series $(L,V)\in G^r_d(C)$ and predicts that the multiplication map
$$\nu(L,V):\Sym^2 V\to H^0\left(C,L^2\right)$$
fails to be of maximal rank for the linear series in a determinantal locus $Q_d^r(C)\subset G^r_d(C)$ which is of the expected dimension 
$$\dim Q^r_d(C)=\rho(g,r,d)-1-\abs{\binom{r+2}{2}-(2d+1-g)}=:\alpha(g,r,d).$$

The following three theorems constitute our contribution to the current knowledge regarding the SMRC. Notice that, when we say a certain scheme $X$ is of the expected dimension $\delta\ge 0$, we also implicitly say it is nonempty. For a better view of the cases presented in \autoref{SMRC r=3}, \autoref{SMRC r=4} and in \autoref{SMRC r ge 5}, we recommend the reader to check the tables in \autoref{r=3}. 

\begin{theorem}
[SMRC for $r=3$]\label{SMRC r=3}
Let $C$ be a general curve of genus $g\ge 1$ and let $d$ be a positive integer such that $\rho(g,3,d)\ge 0$. Then the degeneracy locus in $G^3_d(C)$ defined by 
$$Q^3_d(C)=\{\ell\in G^3_d(C)\mid \nu(\ell)\textnormal{ \textit{is not of maximal rank}}\}$$
is either 

$\bullet$ empty when $(g,d)\in\{(1,4),(2,5),(4,6)\}$, when $d\le g+1$ or when $d=g+2$ and $g\ge 5$ is odd, or

$\bullet$ of the expected dimension $2(d-g-2)$, when $d\ge g+3$ or when $d=g+2$ and $g\ge 6$ is even.
\end{theorem}

Therefore, SMRC holds for $r=3$, except the cases $(g,d)=(2,5)$ and $(g,d)=(2k+1,2k+3)$ for $k\ge 2$, when $Q^3_d(C)$ is predicted to have dimension zero, yet it is empty. 
\vspace{0.2cm}

The next situation $r=4$ already turns out to be much more difficult. For the majority of the cases, we prove the existence of a component of the expected dimension, and we list beforehand the situations in which we know the conjecture to hold in its entirety. 

\begin{theorem}
[SMRC for $r=4$]\label{SMRC r=4}
Let $C$ be a general curve of genus $g\ge 1$ and let $d$ be a positive integer such that $\rho(g,4,d)\ge 0$ and $d\ge g+3$. 
Then the degeneracy locus in $G^4_d(C)$ defined by 
$$Q_d^4(C)=\{\ell\in G^4_d(C)\mid \nu(\ell)\textnormal{ \textit{is not of maximal rank}}\}$$
is either

$\bullet$ empty when $(g,d)\in\{(5,8),(6,9),(7,10),(1,5),(2,6),(3,7)\}$, 

$\bullet$ of dimension $2$ for $(g,d)\in\{(8,11),(4,8)\}$, dimension $3$ for $(g,d)=(5,9)$, dimension $5$ for $(g,d)=(6,10)$, dimension $8$ for $(g,d)=(4,9)$ and dimension $11$ for $(g,d)=(2,8)$, 

$\bullet$ contains a component of the expected dimension $\alpha(g,4,d)>0$ for any other pair $(g,d)$ with $2d\ge g+14$. 
\end{theorem}

Therefore, SMRC for $r=4$ is known to fail in only two instances so far, namely when $(g,d)=(7,10)$ or $(4,8)$, as the expected dimensions are $0$ and $1$ respectively, while we actually get an empty locus in the former case and a $2$-dimensional locus in the latter case. 

A complete solution to SMRC for $r=4$, that is beyond the existence of a component of $Q^4_d(C)$, looks very challenging at the moment, especially for the case when the failure of $Q^4_d(C)$ singles out (virtual) divisors on $\ol{\mrond}_g$, see also the discussion in the end of Section 2. 
\vspace{0.2cm}

When $r\ge 5$, we prove the existence of a component of $Q^r_d(C)$ which is of the expected dimension in a specific range. The result in \autoref{SMRC r ge 5} for $r=4$ serves in the incomplete range as the basis case in the proof by induction on $g,d$ of \autoref{SMRC r=4}. 

\begin{theorem}\label{SMRC r ge 5}
Let $g,d\ge 1$ and $r\ge 4$ be integers such that $d>g+r$, $2d+1-g\ge \binom{r+2}{2}$ and $\alpha(g,r,d)\ge g$. Then, for a general curve $C$ of genus $g$, the locus 
$$Q_{d}^r(C)=\{\ell\in G^r_d(C)\mid \nu(\ell)\textnormal{ \textit{is not of maximal rank}}\}$$
has a component of the expected dimension $\alpha(g,r,d)$ which maps dominantly onto $\Pic^d(C)$. 
\end{theorem}

We mention that Aprodu and Farkas stated the SMRC also for the $n$-fold multiplication maps. However, in this article, we only concern ourselves with the case $n=2$, as it has already been presented above. In what follows, we propose a further refinement, that takes into account the rank $k$ of the quadrics.
\vspace{0.2cm}

We fix positive integers $g,r,d$ such that $r\ge 3$ and $\rho(g,r,d)\ge 0$. Let us also fix a general genus $g$ curve $C$. If we pick a Poincaré line bundle $\lrond$ on $C\times \Pic^d(C)$, we are able to construct vector bundles $\erond$ and $\frond$ over $G^r_d(C)$ such that the fibers over a point $\ell=(L,V)$ are $$\erond{(\ell)}=V\textnormal{ and }\frond{(\ell)}=H^0(C,L^2),$$ as well as a vector bundle morphism $\nu:\Sym^2 \erond\to \frond$ that globalizes the multiplication map $\nu(\ell)$. Notice that $\rk(\Sym^2 \erond)=\dim_\CC \Sym^2 V=\binom{r+2}{2}$ and $\rk(\frond)=h^0(C,L^2)=2d+1-g$. 
\vspace{0.2cm}

For $3\le k\le r+1$, we have the subvariety $\Sigma_k\subseteq \PP(\Sym^2\erond)$ such that the fiber $\Sigma_k(\ell)$ over each point $\ell=(L,V)\in G^r_d(C)$ parametrizes all quadric forms with entries in $V$ that have rank at most $k$. It is known that $\Sigma_k$ is of pure codimension $\binom{r-k+2}{2}$ in $\PP(\Sym^2 \erond)$. 

Let $\PP \ker(\nu)$ be the subvariety of $\PP(\Sym^2 \erond)$ where the map $\PP(\nu):\PP(\Sym^2 \erond)\dashrightarrow \PP(\frond)$ is not well defined. Also set $\Sigma_k(C)=\Sigma_k\cap \PP\ker(\nu)$ be the locus in $\Sigma_k$ where the induced rational map $\nu_k:\Sigma_k\dashrightarrow \PP(\frond)$ is not well defined. 

The fiber of $\Sigma_k(C)$ over each point $\ell\in G^r_d(C)$ is the projective variety $\Sigma_k(C,\ell)$ in $\PP(\Sym^2 V)$ of all quadric forms of rank at most $k$ that lie in $\PP\ker\left(\nu(\ell)\right)$. A naive dimension count based on the observation that $\Sigma_k(C,\ell)=\Sigma_k(\ell)\cap \PP\ker(\nu(\ell))$ sets the lower bound
$$\dim \Sigma_k(C,\ell)\ge \max\left\{-1,\binom{r+2}{2}-\binom{r-k+2}{2}-(2d+2-g)\right\}=:q(g,r,d,k),$$
where dimension $-1$ stands for an empty locus. Under the assumption $r-d+g\le 1$, it is proven in \cite{Kadikoylu19} that $\Sigma_k(C,\ell)$ is of pure dimension $q(g,r,d,k)$, for a general $\ell\in G^r_d(C)$. We propose a finer problem, which is concerned with the fiber over the degeneracy locus of $\nu_k$ for \textit{all} linear series $\ell\in G^r_d(C)$. 

\begin{conjecture}
[Stratified Strong Maximal Rank Conjecture]\label{SSMRC}
Fix integers $g,r,d,k$ as above. For a general curve $C$ of genus $g$, the degeneracy locus 
$$Q_{d,k}^r(C):=\big\{\ell\in G^r_d(C)\mid \Sigma_{k}(C,\ell)\textnormal{ \textit{is of dimension higher than} }q(g,r,d,k)\big\}$$
has dimension 
$$\dim Q_{d,k}^r(C)=\rho(g,r,d)-1-\abs{\binom{r+2}{2}-\binom{r-k+2}{2}-(2d+1-g)}=:\beta(g,r,d,k),$$
where, as above, negative dimension stands for an empty variety. 
\end{conjecture}

We remark that, when the rank $k=r+1$ is maximum possible, we recover the Strong Maximal Rank Conjecture. Also note that, unless $k=r+1$, the number $\beta(g,r,d,k)$ is not an actual dimensional lower bound for $Q_{d,k}^r(C)$, but rather a naive expectation for the dimension, since $Q_{d,k}^r(C)$ is not represented as an actual degeneracy locus for a morphism of vector bundles. Nevertheless, most of the following results confirm this naive expectation, with only a few instances where it fails. 
\vspace{0.2cm}

To make notation a bit easier to follow, set $s=d-g-r$ from now on. Then $s\le -1$ stands for the special range, while $s\ge 1$ means that any $g^r_d$ is necessarily incomplete. 

\begin{theorem}\label{SSMRC rank 3}
The stratified SMRC holds for $k=3$ in the special range, as well as in the nonspecial range for even degree $d$, when either $s=0$ with $g\ge r-1$ or $s\ge 1$ with $g\ge r-2$. Concretely, we have the description below. 
\[
\begin{tabular}{|c|c|}
\hline
$k=3$ & $r\ge 3$ \\
\hline
$s\le -1$ & empty, as expected  \\
\hline
\makecell{$s=0$ for $d$ odd} & \makecell{empty, while expected dimension is $r-2$} \\
\hline
\makecell{$s\ge 1$ for $d$ odd} & \makecell{dimension $(r-1)s$, that is $r-2$ less than expected, \\ for $g\ge r-1$ with $s=1$ or $g\ge r-2$ with $s\ge 2$} \\
\hline
\makecell{$s\ge 0$ for $d$ even} & \makecell{expected dimension $(r-1)(s+1)-1$, for $g$ as in the statement}  \\
\hline
\end{tabular}
\]
\end{theorem}
\vspace{0.2cm}

Next, we discuss \autoref{SSMRC} for quadrics of rank $4$. 

\begin{theorem}\label{SSMRC rank 4}
The stratified SMRC holds for $k=4$ when $r=4$ as well as when $r\ge 5$ and $s\le -2$, $d=g+2$ or $d=g+3$ and $g$ is even. The following table describes the knowledge regarding the other cases as well. 

\[
\begin{tabular}{|c|c|c|}
\hline
$k=4$ & $r=4$ & $r\ge 5$ \\
\hline
$s\le -2$ & \makecell{empty, as expected} & \makecell{empty for $g$ even and $d=g+3$, \\ empty for $d=g+2$, as expected} \\
\hline
$s=-1$ & \makecell{empty for $g=5$ and $g=6$, \\ dimension $1$ for $g\ge 7$, as expected} & \makecell{a component of expected dimension $r-3$ \\ for $g\gg 0$} \\
\hline
$s\ge 0$ & \makecell{expected dimension $3s+4$ \\ for $g\gg 0$} & \makecell{a component of expected dimension \\ $(r-1)(s+2)-2$ \\ for $g\gg 0$} \\
\hline
\end{tabular}
\]
\end{theorem}
\vspace{0.2cm}

The case $s\le -2$ is a consequence of a more general result, which is interesting on its own:

\begin{theorem}\label{sum of 3 pencils}
Let $a\ge 1$. The canonical bundle of a general smooth curve $[C]\in \mrond_{2a}$ cannot be written as a sum of three line bundles $A_0$, $A_1$, $A_2$ with $h^0(C,A_i)\ge 2$, for $i\in\{0,1,2\}$, such that one of them is of minimal degree (i.e. a pencil of degree $a+1$). 
\end{theorem}

As pointed out earlier, if $\ell\in G^r_d(C)$ is generic, then the locus in $\ker\nu(\ell)$ consisting of quadrics of rank at most $k$ is of the expected dimension $q(g,r,d,k)$. Consequently, when $q(g,r,d,k)=-1$, there are no quadrics of rank at most $k$ in $\ker\nu(\ell)$ for the generic $\ell\in G^r_d(C)$. By contrast, \autoref{SSMRC rank 4} tells us in particular that if $q(g,r,d,k)=-1$ still holds, $k\ge 5$ and we furthermore impose $d-g-r=s\ge -1$, then there exist linear series $\ell\in G^r_d(C)$ for which $\ker\nu(\ell)$ contains quadrics of rank at most $k$. Therefore, the generic behavior in $G^r_d(C)$ does not hold everywhere, and $Q^r_{d,k}(C)$ is nonempty as a consequence. Concretely, we have the following result. 

\begin{corollary}\label{existence for Q^r_d,k(C)}
Consider integers $g,r,d,k,s$ such that $r\ge 4$, $5\le k\le r+1$, $d-g-r=s\ge -1$ and $g\ge (k-2)r-\binom{k-1}{2}-2s$. If $C$ is a general smooth curve of genus $g$, then $Q^r_{d,k}(C)$ is nonempty.
\end{corollary}

\begin{acknowledgements*}
I would like to thank my advisor Gavril Farkas for the support and for all the valuable discussions. This research has been carried out as part of the DFG Research Training
Group 2965 ``From geometry to numbers" project number 512730679
involving Humboldt-Universität zu Berlin and Leibniz Universität
Hannover. This work has also been funded by the DFG under Germany's Excellence Strategy – The Berlin
Mathematics Research Center MATH+.
\end{acknowledgements*}

\section{Previous work in the Strong Maximal Rank Conjecture}

In this section, we collect some of the results that are already in the literature, which indirectly imply some cases of the Strong Maximal Rank Conjecture. We start with the famous theorem of Green and Lazarsfeld \cite{GL86}, which states that for any \textit{very ample} line bundle $L$ on an arbitrary smooth curve $C$ of genus $g$ such that 
\begin{align}\label{GL degree bound}
    \deg(L)\ge 2g+1-2h^1(C,L)-\Cliff(C),
\end{align}
the multiplication map $\nu(L)$ is surjective. Using Theorem 4.e.1\cite{Green84}, the same statement can be extended also to the multiplication maps $\Sym^n H^0(C,L)\to H^0\left(C,L^n\right)$. As in the literature this result is often connected to Green and Lazarsfeld's secant conjecture, it is also interesting to see that it has immediate implications in the Strong Maximal Rank Conjecture. 

Concretely, one can easily check that \eqref{GL degree bound} holds only if $h^1(C,L)\in\{0,1\}$. On the other hand, if $-s=r-d+g\in \{0,1\}$, then under the assumption that $\rho(g,r,d)=g+s(r+1)<r+s$, one can check that all $g^r_d$'s on the general genus $g$ curve $C$ are complete and very ample. Therefore, Green and Lazarsfeld's theorem confirms the Strong Maximal Rank Conjecture in the following instances: 

$\bullet$ $r-d+g=1$ with $r+1\le g\le 2r-1$, and 

$\bullet$ $r-d+g=0$ with $1\le g\le r-1$. 
\\
There is one single exception to this matter, namely when $(g,r,d)=(7,4,10)$: the expected dimension $\alpha(7,4,10)=0$ is zero, yet the above discussion just proved that $Q^4_{10}(C)=\emptyset$ for genus $7$ curves $C$. 
\vspace{0.2cm}

In the other range $\rho(g,r,d)=g+s(r+1)\ge r+s$, Green and Lazarsfeld's result still gives information when $s=0$ and $r\le g\le 2r-3$. In other words, it is still true that any very ample $g^r_d$ on a genus $g$ curve $C$ has a surjective multiplication map $\nu(\ell)$. This time, however, Theorem 0.5\cite{Farkas08} shows that there exist non-very ample $g^r_d$'s on the general genus $g$ curve $C$ and, furthermore, we have the following dimensionality statement
\begin{align}\label{dim of non very ample}
    \dim\left\{\left(\ell,(x,y)\right)\in G^r_d(C)\times C_2\mid \dim \ell(-x-y)\ge r-1\right\}=\rho(g,r,d)-r+2.
\end{align}
A non-very ample linear series $\ell\in G^r_d(C)$ has a non-surjective multiplication map $\nu(\ell)$. Therefore, for $r\le g\le 2r-3$ with $s=0$, the variety $Q^r_d(C)$ is simply parametrizing the non-very ample linear series. It follows from \cite{GHS02} that, if the curve $C$ is assumed to be general, then for the generic point $(\ell,(x,y))$ in \eqref{dim of non very ample} the series $\ell$ induces a generically injective map from $C$ to $\PP^r$, hence there are finitely many pairs $(x,y)\in C_2$ to which it corresponds. Consequently, we obtain $\dim Q^r_d(C)=\rho-r+2$ in this very specific area. For $r=4$ (see also \autoref{r=4}), this shows for the generic curve $C$ that

$\bullet$ if $(g,d)=(4,8)$, then $\dim Q^4_8(C)=2$ and yet the expected dimension is $\alpha(4,4,8)=1$, 

$\bullet$ if $(g,d)=(5,9)$, then $\dim Q^4_9(C)=3$ and this is also the expected dimension $\alpha(5,4,9)=3$.
\\
We close this circle of ideas by also mentioning \eqref{dim of non very ample} gives an existence statement for $Q^r_d(C)$ in all the other instances of $(g,r,d)$ with $r-d+g\in\{0,1\}$ and $\binom{r+2}{2}>2d+1-g$, by simply considering any non-very ample $g^r_d$. 

In \cite{CGZ21} it is also shown that, when $\binom{r+2}{2}>2d+1-g$, that is the surjective area, the class of $Q^r_d(C)$ is positive in a vast range. In particular, in the same range the nonemptiness of $Q^r_d(C)$ is ensured. As it is shown in the same article, this range has important connections to the Bertram-Feinberg-Mukai conjecture \cite{BF98}, \cite{Mukai97}. 
\vspace{0.2cm}

On the other hand, there are nontrivial results that consider the injective range as well, namely when $\binom{r+2}{2}\le 2d+1-g$ and, moreover, the expected dimension $\alpha(g,r,d)<0$ is negative. Concretely, when $C$ is a general curve of genus $g$, 

$\bullet$ in \cite{Teixidor03} it is shown via degeneration methods that $\nu(\ell)$ is injective for all linear series $\ell\in G^r_d(C)$ with $d\le g+1$, 

$\bullet$ in \cite{FJP25}/\cite{LOTZ24} it is shown via tropical methods/degeneration methods that $\nu(\ell)$ is injective for all linear series $\ell\in G^r_d(C)$ when $(g,r,d)\in \left\{(22,6,25),(23,6,26)\right\}$; this particular result has important applications to the geometry of the moduli spaces of curves, as it is shown in the former article that $\mrond_{22}$ and $\mrond_{23}$ are of general type, 

$\bullet$ in \cite{FO11}, \cite{BF18}, \cite{FP05} and \cite{FV14} it is shown that when $d=g+2$, and either $g\ge 11$ is odd or $g\in\{10,12\}$, then $\nu(\ell)$ is injective for all linear series $\ell\in G^4_d(C)$; for $g=12$, this result gives a divisor in $\ol{\mrond}_{12}$ which is used to prove that the moduli space of odd spin curves $\ol{\srond}_{12}^{-}$ is of general type, and

$\bullet$ in \cite{Farkas09} it is shown that SMRC holds when $g=s(2s+1)$, $r=2s$ and $d=2s(s+1)$, for all $s\ge 2$.

\section{Quadrics of rank $3$ and $4$ and the Strong Maximal Rank Conjecture in $\PP^3$}

\subsection{Quadrics of rank $3$ and $4$ and linear series}

We begin with the following characterization of linear series $\ell\in G^r_d(C)$ that contain a quadric of rank at most $4$ in the kernel of their associated multiplication map $\nu(\ell)$. This will be used recurrently throughout the text. 

\begin{lemma}\label{character of rank 3 and 4}
Let $(L,V)$ be a $g^r_d$ on a genus $g$ curve $C$ and let $B$ be its base locus. Then there is a one to one correspondence between quadrics of rank at most $4$ in $\PP\ker\nu (L,V)$ and triplets $(\ell_1,\ell_2,F)$ consisting of basepoint free pencils $\ell_1=(A_1,V_1)$, $\ell_2=(A_2,V_2)$ and an effective divisor $F\ge B$ such that $L(-F)=A_1\otimes A_2$ and $V(-F)$ contains the image of $V_1\otimes V_2$ via the multiplication map $V_1\otimes V_2\to H^0(C,A_1\otimes A_2)\cong H^0(C,L(-F))$. 

Furthermore, if we look at the quadrics of rank $3$ in $\PP \ker\nu(L,V)$, then we also require $\ell_1=\ell_2$ in the above. 
\end{lemma}
\begin{proof}
Let $\varphi=\varphi_{\ell(-B)}:C\to \PP^r$ be the map induced by $\ell=(L,V)$. We know that $\PP\ker\nu(L,V)$ parametrizes all quadrics in $\PP^r$ that contain the image curve $\varphi(C)$. 

If $Q$ is a quadric of rank $4$ that contains $\varphi(C)$, then $Q$ has two rulings, each of which cuts a pencil on $\varphi(C)$. If $\ell_1=(A_1,V_1)$ and $\ell_2=(A_2,V_2)$ are these two distinct pencils, then they are basepoint free and we can write $L(-B)=A_1\otimes A_2\otimes \OO_C(F')$, for some effective divisor $F'$ on $C$. If we set $F=F'+B\ge B$, we obtain the direct association. 

Conversely, consider any two distinct basepoint free pencils $\ell_1=(A_1,V_1)$ and $\ell_2=(A_2,V_2)$ and any $F\ge B$ such that $L(-F)=A_1\otimes A_2$ and $V(-F)$ contains the image of $V_1\otimes V_2$ via the multiplication map $V_1\otimes V_2\to H^0(C,A_1\otimes A_2)\cong H^0(C,L(-F))$. Let $\{\sigma_0,\sigma_1\}$ and $\{\tau_0,\tau_1\}$ be bases for $V_1$ and $V_2$ respectively, and also let $s$ be a global section of $\OO_C(F)$. Then 
$$ s\sigma_0\tau_0\otimes s\sigma_1\tau_1-s\sigma_0\tau_1\otimes s\sigma_1\tau_0\in \ker\nu(L,V)$$
is a nonzero quadric form in the kernel and it has rank $4$. This settles the other association and it can be easily seen that it is the inverse of the direct one. The same constructions work for the quadrics of rank $3$, with the extra remark that rank $3$ quadrics have a single ruling and hence the associated pencils are, in fact, equal. This establishes the lemma. 
\end{proof}

\subsection{The incidence variety of linear series and rank $3$ quadrics}

The lemma above already allows us to make a clear dimensionality statement regarding linear series $\ell\in G^r_d(C)$ that contain rank $3$ quadrics in $\ker\nu(\ell)$. 

\begin{theorem}\label{incidence variety for rank 3}
Let $g,r,d\ge 1$ be integers such that $r\ge 3$ and $\rho(g,r,d)\ge 0$. For a general curve $C$ of genus $g$, consider the incidence variety 
$$\irond_C^3=\left\{(\ell,q)\mid \ell\in G^r_d(C)\textnormal{ and }q\in \Sigma_3(C,\ell)\right\}$$
and set $s=d-g-r$. Then 

$\bullet$ if $s\le -1$ or if $s=0$ and $d$ is odd, $\irond_C^3$ is empty, 

$\bullet$ if $s\ge 0$ and $d$ is even, $\dim \irond_C^3=(r-1)(s+1)-1$, and 

$\bullet$ if $s\ge 1$ and $d$ is odd, $\dim \irond_C^3=(r-1)s$. 
\end{theorem}
\begin{proof}
According to \autoref{character of rank 3 and 4}, we have the alternative description
$$\irond_C^3=\{\textnormal{triplets }(F,(A,W),V(-F))\textnormal{ satisfying condition }(*)\},$$
where condition $(*)$ is the following four requirements: 

\textbf{(1)} $F\in C_f$, 

\textbf{(2)} $(A,W)\in G^1_a(C)$ is a basepoint free pencil, 

\textbf{(3)} $V(-F)\in G(r+1,H^0(A\otimes A))$ contains the image of $\Sym^2 W\to H^0(C,A\otimes A)$, and 

\textbf{(4)} $2a+f=d$. 
\vspace{0.2cm}

Let us count the dimension of $\irond_C^3$ using the latter description. First, since $C$ is Petri general, we have $h^0(C,A\otimes A)=2a+1-g$ according to Riemann-Roch. This requires, in particular, that $r+1=\dim_\CC V(-F)\le h^0(C,A\otimes A)=2a+1-g$, so $r\le 2a-g$. Consequently, $r\le d-g$ must hold and $\irond_C^3=\emptyset$ if $s\le -1$. 

Suppose henceforth that $s\ge 0$. If $d$ is odd, notice that $d=2a+f$ requires $f\ge 1$, so any valid linear series has at least one basepoint. Observe that $r\le 2a-g\le d-g-1$ must hold then, so the situation $s=0$ for $d$ odd gives an empty locus, $\irond_C^3=\emptyset$. For the rest of the cases, both $d$ odd with $s\ge 1$ and $d$ even with $s\ge 0$, with the remark that the Grassmannian $G(r-2,H^0(C,A\otimes A)/\Sym^2 W)$ parametrizes the spaces $V(-F)$ in question, we can simply count the moduli in order to get 
$$\underbrace{f}_{\textnormal{moduli of }F}+\underbrace{\rho(g,1,a)}_{\textnormal{moduli of }(A,W)}+\underbrace{(r-2)(2a-g-r)}_{\textnormal{moduli of the }r-2\textnormal{ other sections of }V(-F)}=(r-1)(s+1)-1-(r-2)f.$$
For $d$ even, $f\ge 0$ can be anything, while for $d$ odd, $f$ can be anything strictly positive. All these ensure that
$$\dim \irond_C^3=\begin{cases}
    (r-1)(s+1)-1,&\textnormal{ for }d\textnormal{ even and }s\ge 0, \\
    (r-1)s,&\textnormal{ for }d\textnormal{ odd and }s\ge 1.
\end{cases}$$
\end{proof}

\subsection{The proof of \autoref{SSMRC rank 3}}

The emptiness statements when $s\le -1$ or $s=0$ with $d$ odd follow directly from \autoref{incidence variety for rank 3}. As for the remaining statements, we concern ourselves only with the even $d$ case. Then odd $d$ follows from the even case, by adding any basepoint to the linear series in $Q^r_{d-1,3}(C)$. So let $d=2a$ and keep in mind the description of $\irond_C^3$ from above. 

We consider the forgetful map $\pi:\irond_C^3\to Q^r_{d,3}(C)$. The map $\pi$ is well defined in the specified range of the genus $g$ i.e. $g\ge r-1$ for $s=0$ and $g\ge r-2$ for $s\ge 1$, as then $q(g,r,d,3)=-1$ holds and the special behavior of $Q^r_{d,3}(C)$ in $G^r_d(C)$ is given by linear series $\ell=(L,V)$ for which the multiplication map $\nu(\ell):\Sym^2 V\to H^0\left(C,L^2\right)$ contains at least a quadric of rank $3$ in the kernel.

The map $\pi:\irond_C^3\to Q^r_{d,3}(C)$ is obviously surjective and already sets the desired dimensional upper bound $\dim Q^r_{d,3}(C)\le \dim \irond_C^3$. We show that, in the specified range of the genus $g$, there is equality by providing a point $(L,V)\in Q^r_{d,3}(C)$ in the image of a component of the top dimension over which the fiber $\pi^{-1}(L,V)$ is finite. This suffices in order to conclude that $\pi:\irond_C^3\to Q^r_{d,3}(C)$ is generically finite (at least over one component of the target) and thus $\dim Q^r_{d,3}(C)=\dim \irond_C^3$. 
\vspace{0.2cm}

If $a\le g+1$, then the generic $A\in W^1_{a}(C)$ satisfies $h^0(C,A)=2$ and for any other $B\in W^1_a(C)$ such that $A\otimes A=B\otimes B$, if there are any, the condition $h^0(C,B)\le 2$ holds. Indeed, one only needs to observe that 

$\bullet$ if $A\otimes A=B\otimes B$, then $A\otimes B^{-1}$ is a square root of $\OO_C$ and there are finitely many of these ($2^{2g}$ to be precise), and

$\bullet$ the condition $h^0(C,B)\ge 3$ is special in $W^1_a(C)$ i.e. $W^2_a(C)$ is properly contained in $W^1_a(C)$. 
\vspace{0.2cm}
\\
With these observations, let $A\in W^1_a(C)$ generic as above, let $W=H^0(C,A)$ so $(A,W)\in G^1_a(C)$ is complete and consider $L=A\otimes A$. For any vector space $V\in G(r+1,H^0(C,L))$ that contains the image of $\Sym^2 W$, a rank $3$ quadric in the kernel of $\nu(L,V)$ different from the one induced by $(A,W)$ would correspond to some $(B,W')\in G^1_a(C)$ with $A\otimes A=B\otimes B$ and $\Sym^2W'\injto V$. The generality of $A$ shows that there are finitely many such $B$'s, and each of them produces exactly one rank $3$ quadric, since $\dim W'=h^0(C,B)=2$ must hold. 
\vspace{0.2cm}

Now consider the second situation $a\ge g+2$ with $g\ge r-2$. If $L\in \Pic^{2a}(C)$ is a generic line bundle, we infer from \cite{Kadikoylu19} that $\Sigma_3(C,L)$ has the expected dimension $q(g,2a-g,2a,3)=2a-2g-2$. Given the alternative description of rank $3$ quadrics via \autoref{character of rank 3 and 4}, this means that $\PP H^0(C,L)\cong \PP^{2a-g}$ contains a $(2a-2g-2)$-dimensional family of $\PP^2$'s. If we ensure for some vector space $V\in G(r+1,H^0(C,L))$ that $\PP V$ intersects this family in finitely many of these $\PP^2$'s, then we are done. But this is immediate, after we fix one such $\PP^2\cong \PP\Sym^2 W$: since $g\ge r-2$, we know that $r+\left((2a-2g-2)+2\right)\le (2a-g)+2$, so the generic $\PP V$ containing $\PP\Sym^2 W$ intersects the family $\Sigma_3(C,L)\times \PP^2\subseteq \PP H^0(C,L)\cong \PP^{2a-g}$ in an at most $2$-dimensional locus. Consequently, there are finitely many other possible $\PP^2$'s from this family that are contained in $\PP V$, as desired. 

\subsection{The Strong Maximal Rank Conjecture in $\PP^3$}

We prove that SMRC holds in $\PP^3$, with the exceptions of an explicit list of cases, when the expected dimension is $\alpha(g,3,d)=0$ and yet the degeneracy locus is empty. Suppose $C$ is a general curve of genus $g\ge 1$. The gonality of the curve is $\gon(C)=\ceiling {g/2}+1$ and its Clifford index is $\Cliff(C)=\floor{\frac{g-1}{2}}$. 

\begin{proof}
[Proof of \autoref{SMRC r=3}]
Let $\ell=(L,V)$ be a $g^3_d$ on the curve and consider the associated multiplication map $$\nu(\ell):\Sym^2 V\to H^0\left(C,L^2\right).$$ 
Under the generality hypothesis $\rho(g,3,d)\ge 0$, one can easily check that the only instances when $10=\dim_\CC \Sym^2V\ge h^0\left(C,L^2\right)=2d+1-g$ are $(g,d)\in\{(1,4),(2,5),(4,6)\}$. These three cases fall in the situation described at \eqref{GL degree bound} in Section 2, so SMRC is true with the following exception: for $(g,d)=(2,5)$, the expected dimension is $\alpha(2,3,5)=0$, while the degeneracy locus in $G^3_5(C)$ is actually empty. 
\vspace{0.2cm}

All the remaining cases for quadrics and $r=3$ fall in the injective area, so $\dim_\CC\Sym^2 V=10\le 2d+1-g=h^0\left(C,L^2\right)$. Therefore the degeneracy locus $Q^3_d(C)$ parametrizes linear series $\ell=(L,V)\in G^3_d(C)$ for which the multiplication map $\nu(\ell):\Sym^2 V\to H^0\left(C,L^2\right)$ fails to be injective. Consider the incidence variety
$$\irond:=\left\{(\ell,q)\mid \ell\in Q^3_d(C)\textnormal{ and }q\in \PP \ker \nu(\ell)\right\}$$
which has the natural forgetful map $\pi:\irond\to Q^3_d(C)$. This is surjective, by definition. Since we work in $\PP^3$, the quadric $q\in \PP\ker \nu(\ell)$ has rank at most $4$. Both situations $\rk(q)=3$ and $\rk(q)=4$ fall in the hypotheses of \autoref{character of rank 3 and 4}, so we have the alternative description
$$\irond=\{\textnormal{ quadruplets }(B,(A_1,W_1),(A_2,W_2),V(-B))\textnormal{ satisfying condition }(*)\},$$
where condition $(*)$ means the following four requirements: 

\textbf{(1)} $B\in C_b$, 

\textbf{(2)} $(A_1,W_1)\in G^1_{a_1}(C)$ and $(A_2,W_2)\in G^1_{a_2}(C)$ are basepoint free pencils, 

\textbf{(3)} $V(-B)\in G(4,H^0(A_1\otimes A_2))$ contains the image of $W_1\otimes W_2\to H^0(C,A_1\otimes A_2)$, and 

\textbf{(4)} $a_1+a_2+b=d$. 
\vspace{0.2cm}

The situation when the quadric $q$ is of rank $3$ has been discussed in \autoref{SSMRC rank 3}. This sublocus is nonempty only if $d\ge g+3$, in which case there is an upper bound of 
$$2(d-g-2)-1$$
on its dimension. So let us focus on the situation when the quadric $q$ is of rank $4$. This means that the two pencils $(A_1,W_1)$ and $(A_2,W_2)$ have to be distinct, so the image of $W_1\otimes W_2\to H^0(C,A_1\otimes A_2)$ is $4$-dimensional. As a consequence, $V(-B)$ is uniquely determined by requirement \textbf{(3)}. Now we can simply count the moduli in order to get

$$\underbrace{b}_{\textnormal{moduli of }B}+\underbrace{\rho(g,1,a_1)}_{\textnormal{ moduli of }(A_1,W_1)}+\underbrace{\rho(g,1,a_2)}_{\textnormal{moduli of }(A_2,W_2)}=2(d-g-2)-b.$$
Since $b\ge 0$ can be anything, we deduce that the moduli for $\rk(q)=4$ is overall $2(d-g-2)$. Finally, the surjectivity of $\pi:\irond\to Q^3_d(C)$ implies the bound $\dim Q^3_d(C)\le \dim \irond=2(d-g-2)$. Since $Q^3_d(C)$ itself has a determinantal structure that imposes the lower bound $\dim Q^3_d(C)\ge 2(d-g-2)$ on all its irreducible components, we finally deduce that $Q^3_d(C)$ is of pure dimension $2(d-g-2)=\alpha(g,3,d)$.

Notice that $a_1\ge \gon(C)$ as well as $a_2\ge \gon(C)$ are required, so as soon as $d\ge 2\gon(C)$, then both $\irond$ and $Q^3_d(C)$ are nonempty. This is almost always in tune with the above dimension count, and the only exception is when $d=g+2$ with $g$ odd, in which case the expected dimension is $0$ and yet $d<2\gon(C)$, so the degeneracy locus $Q^3_d(C)$ is, in fact, empty. 
\end{proof}

\section{On the Unrepresentability of the Canonical Bundle as a Sum of $3$ Pencils}

In this section, we prove the $s\le -2$ part of \autoref{SSMRC rank 4}. The main result is the following 

\begin{theorem}\label{no quads of rk 4 generically}
Let $g,d\ge 1$ be integers with $\ceiling{g/2}+1\le d\le g-1$ and let $C$ be a general curve of genus $g$. For a general $A\in W^1_d(C)$, the multiplication map 
$$\nu\left(\omega_C\otimes A^{-1}\right):\Sym^2 H^0\left(C,\omega_C\otimes A^{-1}\right)\to H^0\left(C,\omega_C^2\otimes A^{-2}\right)$$
contains no quadrics of rank at most $4$ in the kernel. 
\end{theorem}

We begin with the following easy remark. If $\ker\nu\left(\omega_C\otimes A^{-1}\right)$ were to contain a quadric of rank at most $4$, then we would be able to write $\omega_C=A\otimes A_1\otimes A_2$ for some pencils $A_1\in W^1_{a_1}(C)$ and $A_2\in W^1_{a_2}(C)$, where $2g-2=d+a_1+a_2$. Since $C$ is (Brill-Noether) general, any such pencils must have $a_1,a_2\ge \ceiling{g/2}+1$, so we would need $2g-2\ge 3\ceiling{g/2}+3$ i.e. $g=10$ or $g\ge 12$. Therefore, the small genus cases are true for \textit{any} line bundle $A\in W^1_d(C)$. In what follows, we consider only the instances $g=10$ or $g\ge 12$. 

Since the Hurwitz scheme $\hrond_{g,d}$ is irreducible, we see that it suffices to prove the theorem for a single instance of a pair $(C,A)\in \hrond_{g,d}$. 

\begin{proof}
[Construction of $(C,A)$]
This is inspired from the proof of Theorem 12.5 in \cite{FR20}. Consider a polarized $K3$ surface $(S,H)$ with Picard rank $2$ and let $\Pic(S)=\ZZ\cdot H\oplus \ZZ\cdot E$, where $H^2=2g-2$, $E^2=0$ and $H\cdot E=d$. Take a general element $C\in \abs{H}$ and set $A:=\OO_C(E)$. Then $C$ satisfies the Brill-Noether theorem and $A$ is a degree $d$ pencil on $C$. 

Our first claim is that $S$ contains no $(-2)$-curves. Indeed, if $X$ were one such, we could write $X=\alpha H+\beta E$ in the Picard lattice, for some integers $\alpha$ and $\beta$. Then 
$$-2=(X)^2=(\alpha H+\beta E)^2=\alpha^2 (2g-2)+2\alpha \beta d\lthen -1=\alpha^2(g-1)+\alpha\beta d.$$
This would force $\alpha$ to be either $1$ or $-1$; switching signs if necessary, one could assume $\alpha=1$, and then would get $0=g+d\beta$, implying $d$ divides $g$, which is a contradiction. 

Observe now that $(H-2E)^2=2g-2-4d<0$. Since there are no $(-2)$-curves on $S$, it follows that any effective divisor on $S$ is nef and has nonnegative self intersection; in particular, we deduce that $h^0(S,H-2E)=0$. 

Now look at the short exact sequence 
$$0\to \OO_S(H-2E)\to \OO_S(2H-2E)\to \OO_C(2K_C-2A)\to 0,$$
of which a piece of the long exact sequence in cohomology is 
$$0\to H^0(\OO_S(2H-2E))\to H^0(\OO_C(2K_C-2A))\to H^1(\OO_S(H-2E))\to \dots$$
We deduce the natural injection $H^0(\OO_S(2H-2E))\injto H^0(\OO_C(2K_C-2A))$. Similarly, using that $E$ is an elliptic pencil and thus $h^0(S,E)=2$, $h^1(S,E)=0$ and $h^2(S,E)=0$, taking cohomology in the short exact sequence 
$$0\to \OO_S(-E)\to \OO_S(H-E)\to \OO_C(K_C-A)\to 0$$
gives the natural isomorphism $H^0(\OO_S(H-E))\cong H^0(\OO_C(K_C-A))$. We have a commutative diagram 
\[\begin{tikzcd}
	0 & {I_{S,H-E}(2)} & {\Sym^2H^0(\OO_S(H-E))} & {H^0(\OO_S(2H-2E))} \\
	0 & {I_{C,K_C-A}(2)} & {\Sym^2H^0(\OO_C(K_C-A))} & {H^0(\OO_C(2K_C-2A))}
	\arrow[from=1-1, to=1-2]
	\arrow[from=1-2, to=1-3]
	\arrow["\cong", from=1-2, to=2-2]
	\arrow[from=1-3, to=1-4]
	\arrow["\cong", from=1-3, to=2-3]
	\arrow[hook, from=1-4, to=2-4]
	\arrow[from=2-1, to=2-2]
	\arrow[from=2-2, to=2-3]
	\arrow[from=2-3, to=2-4]
\end{tikzcd}\]
which tells us that it suffices to show that the multiplication map 
$$\Sym^2 H^0(\OO_S(H-E))\to H^0(\OO_S(2H-2E))$$
contains no rank $4$ elements in its kernel. Assuming it does, this means that there exist divisor classes $E_1$ and $E_2$ on $S$ with $H-E=E_1+E_2$, $h^0(S,E_1)\ge 2$ and $h^0(S,E_2)\ge 2$, for which we can write 
$$\begin{cases}
    E_1=\alpha_1 H+\beta_1E\\
    E_2=\alpha_2 H+\beta_2 E
\end{cases},$$
where $\alpha_i,\beta_i$ are integers. Notice that $E$ is nef and $E\cdot E_1=\alpha_1d\ge 0$, so $\alpha_1\ge 0$ and similarly $\alpha_2\ge 0$. Moreover, we must have $\alpha_1+\alpha_2=1$, so let us assume $\alpha_1=1$ and $\alpha_2=0$. This gives $E_2=\beta_2 E$, whence $\beta_2\ge 1$, which together with $\beta_1+\beta_2=-1$ implies $\beta_1\le -2$. But, at the same time, this would mean $2\le h^0(S,E_1)\le h^0(S,H-2E)=0$, and that is a contradiction. 
\end{proof}

A particular range when the genericity of $A\in W^1_d(C)$ in the above \autoref{no quads of rk 4 generically} can be turned into a statement regarding \textit{all} pencils $A\in W^1_d(C)$ is when the genus $g=2a$ is even and $d=a+1$ is the minimum degree. Indeed, when $C$ is general, $W^1_{a+1}(C)$ is reduced and finite, so a monodromy argument together with, again, the irreducibility of the Hurwitz scheme $\hrond_{2a,a+1}$ allow to pass the same statement regarding the inexistence of quadrics of rank at most $4$ to all pencils in $W^1_{a+1}(C)$:

\begin{corollary}\label{no quads of rk 4 if minimal pencil}
Let $a\ge 1$ and let $C$ be a general curve of genus $2a$. For any $A\in W^1_{a+1}(C)$, the multiplication map 
$$\nu\left(\omega_C\otimes A^{-1}\right):\Sym ^2 H^0\left(C,\omega_C\otimes A^{-1}\right)\to H^0\left(C,\omega_C^2\otimes A^{-2}\right)$$
contains no quadrics of rank at most $4$ in the kernel. 
\end{corollary}

Having in mind the alternative description of the existence of a quadric of rank $4$ in the kernel of the multiplication map, we also obtain \autoref{sum of 3 pencils}:

\begin{corollary}\label{sum of 3 pencils with a minimal}
Let $a\ge 1$. The canonical bundle of a general smooth curve $C$ of genus $2a$ cannot be written as a sum $\omega_C=A_1\otimes A_2\otimes A_3$ where $A_i$ are line bundles with $h^0(C,A_i)\ge 2$ for $i\in \{0,1,2\}$, such that one of them is of minimal degree (i.e. a pencil of degree $a+1$). 
\end{corollary}

Finally, we address the implications of this discussion in \autoref{SSMRC} for $k=4$. Let again $C$ be a general curve of genus $g$ and $(L,V)\in G^r_d(C)$ be any line bundle on it, where $r,d\ge 1$ are integers such that $g-d+r\ge 2$. \autoref{SSMRC} predicts that $\ker\nu(L,V)$ contains no quadrics of rank $4$. Since $\dim\abs{L}-d+g\ge r-d+g\ge 2$ and also since $\nu(L,V)$ factors through $\nu(L)$, 
\[\begin{tikzcd}
	{\Sym^2 V} & {\Sym^2 H^0(C,L)} & {H^0\left(C,L^2\right),}
	\arrow[hook, from=1-1, to=1-2]
	\arrow["{\nu(L,V)}"', curve={height=18pt}, from=1-1, to=1-3]
	\arrow["{\nu(L)}", from=1-2, to=1-3]
\end{tikzcd}\]
we see that it suffices to check this statement when $V=H^0(C,L)$ i.e. the linear series is complete. We do that in the following two instances. 
\vspace{0.2cm}

$\bullet$ When $d\le g+2$, if $\ker\nu(L)$ were to have a quadric of rank $4$, we would be able to write $L=A_1\otimes A_2$ for some pencils $A_1\in W^1_{a_1}(C)$ and $A_2\in W^1_{a_2}(C)$, where $a_1+a_2=d$. But, again, the generality of $C$ requires $a_1,a_2\ge \ceiling{g/2}+1$, so the bound $d\le g+2$ forces $g=2a$ be even, $a_1=a_2=a+1$ and $d=2a+2$. Since $h^0\left(\omega_C\otimes L^{-1}\right)\ge 2$, we derive a contradiction from \autoref{sum of 3 pencils with a minimal}. 

$\bullet$ When $d\le g+3$ and $g=2a$ is even, if $\ker\nu(L)$ were to have a quadric of rank $4$, the same argument would force $d=2a+3$ and $L=A_1\otimes A_2$ for some $A_1\in W^1_{a+1}(C)$ and $A_1\in W^1_{a+2}(C)$. This is a contradiction according to the same \autoref{sum of 3 pencils with a minimal}. 

\section{The Stratified Strong Maximal Rank Conjecture for Rank $4$ and $s\ge -1$}

\subsection{Varieties of secant divisors}

We begin this section by recalling some useful results concerning the varieties of secant divisors to a curve. Let $C$ be a smooth curve of genus $g$ and consider a linear series $\ell=(L,V)\in G^r_d(C)$. For integers $0\le f<e$, we have the determinantal variety 
$$V^{e-f}_e(\ell)=\{D\in C_e\mid \dim \ell(-D)\ge r-e+f\}$$
that parametrizes effective divisors of degree $e$ on $C$ that impose at most $e-f$ conditions on the linear series $\ell$. The determinantal structure tells us that every irreducible component of $V^{e-f}_e(\ell)$ is of dimension at least $e-f(r+1-e+f)$. In \cite{Farkas08}, the upper dimensional bound 
$$\dim \{(\ell,D)\in G^r_d(C)\times C_e\mid D\in V^{e-f}_e(\ell)\}\le \rho(g,r,d)+e-f(r+1-e+f)$$
is proven, under the hypothesis that $C$ is of general moduli. As a consequence, if $V^{e-f}_e(\ell)$ is nonempty for the general series $\ell\in G^r_d(C)$, then $V^{e-f}_e(\ell)$ is of pure dimension $e-f(r+1-e+f)$ for the general $\ell\in G^r_d(C)$. 

The upper bound mentioned above also gives the following
\begin{proposition}\label{super special secant divisors}
Let $h\ge 0$ be an integer and $C$ a general curve of genus $g$. Then
$$\dim\left\{(\ell,\OO_C(D))\mid \ell\in G^r_d(C),\textnormal{ }D\in V^{e-f}_e(\ell)\textnormal{ and }h^0(C,D)\ge h+1\right\}\le$$
$$\le \rho(g,r,d)+e-f(r+1-e+f)-h.$$
\end{proposition}
\begin{proof}
One only needs to look at the map that sends an $e$-secant divisor $D$ that imposes at most $e-f$ conditions on a linear series $\ell$ to its corresponding line bundle $\OO_C(D)$. This is 
$$\left\{(\ell,D)\in G^r_d(C)\times C_e\mid D\in V^{e-f}_e(\ell)\right\}\longrightarrow \left\{(\ell,\OO_C(D))\mid \ell\in G^r_d(C)\textnormal{ and }D\in V^{e-f}_e(\ell)\right\}.$$
Over the sublocus where $h^0(C,D)\ge h+1$, the fibers of this map contain the complete linear series $\abs{D}$, so they are all at least $h$-dimensional. This settles the claimed upper bound. 
\end{proof}

\subsection{Special tensors of pencils}

Having in mind the alternative description of rank $4$ quadrics that appear in the kernel of the multiplication maps given by \autoref{character of rank 3 and 4}, we continue with the analysis of the varieties $W^1_{a_1}(C)\times W^1_{a_2}(C)$, where $C$ is a general genus $g$ curve and $\ceiling {g/2}+1\le a_1\le a_2\le g$ are integers that give a splitting $d=a_1+a_2$. From the Riemann-Roch theorem, we have 
$$h^0(C,A_1\otimes A_2)-h^0\left(C,\omega_C\otimes A_1^{-1}\otimes A_2^{-1}\right)=d+1-g.$$
For any line bundle $A_1\in W^1_{a_1}(C)$, we can choose arbitrary points $p_1,p_2,\dots,p_{a_2-a_1}$ on the curve $C$ and set $A_2=A_1(p_1+\dots+p_{a_2-a_1})$. For such a pair $(A_1,A_2)$, we have $h^0\left(C,\omega_C\otimes A_1^{-1}\otimes A_2^{-1}\right)=0$ since $C$ is (Petri) general. Therefore, the minimum value $h^0(C,A_1\otimes A_2)=d+1-g$ is attained in all irreducible components of $W^1_{a_1}(C)\times W^1_{a_2}(C)$. 
\vspace{0.2cm}

The natural question to follow revolves around the stratification of  $W^1_{a_1}(C)\times W^1_{a_2}(C)$ with respect to the value of $h^0(C,A_1\otimes A_2)$ when $A_1\in W^1_{a_1}(C)$ and $A_2\in W^1_{a_2}(C)$. \autoref{deg locus in sum of pencils} below gives a complete description of the first stratum, that is when $h^0(C,A_1\otimes A_2)\ge d+2-g$ or equivalently $h^0\left(C,\omega_C\otimes A_1^{-1}\otimes A_2^{-1}\right)\ge 1$ attains the next possible value. The second stratum, the locus where $h^0\left(C,\omega_C\otimes A_1^{-1}\otimes A_2^{-1}\right)\ge 2$, is expected to be empty and we proved this expectation when the genus $g=2a$ is even and $a_1=a+1$ is minimal (see the previous section). 

\begin{theorem}\label{deg locus in sum of pencils}
Let $g,r,d\ge 1$ be integers such that $r\ge 3$, $d=g+r-1$ and $\rho(g,r,d)\ge 0$. Let $C$ be a general curve of genus $g$ and consider any splitting $d=a_1+a_2$, where $\ceiling {g/2}+1\le a_1\le a_2\le g$. It is known that $G^1_{a_1}(C)\times G^1_{a_2}(C)$ is a variety of dimension $2(r-3)$. Then the locus 
$$\grond_{C,a_1,a_2}=\{\left((A_1,V_1),(A_2,V_2)\right)\in G^1_{a_1}(C)\times G^1_{a_2}(C)\mid h^0(C,A_1\otimes A_2)\ge r+1\}$$
is nonempty and of dimension $r-3$. Furthermore,

$\bullet$ if $\rho(g,1,a_1)\ge 1$, then $\grond_{C,a_1,a_2}$ contains a component $\Sigma$ of dimension $r-3$ that dominates $G^1_{a_1}(C)$ and whose general point $\left((A_1,V_1),(A_2,V_2)\right)\in \Sigma$ satisfies $h^0(C,A_1\otimes A_2)=r+1$, and 

$\bullet$ if $\rho(g,1,a_1)=0$ i.e. $g=2a$ is even and $a_1=a+1$, then $\grond_{C,a+1,a_2}$ surjects onto $G^1_{a+1}(C)$ and the fiber over each $(A_1,V_1)\in G^1_{a+1}(C)$ is a locus of pure dimension $r-3$ in $G^1_{a_2}(C)$, such that for every point $(A_2,V_2)$ in this locus we have $h^0(C,A_1\otimes A_2)=r+1$. 
\end{theorem}

The situation $r=3$ can only hold when the genus is even, $g=2a$. In this case, $a_1=a_2=a+1$, $d=2a+2$ and the generality of $C$ implies $G^1_{a+1}(C)=W^1_{a+1}(C)$ is reduced and finite, with a total of $w^1_{a+1}=\frac{1}{a+1}\binom{2a}{a}$ points. We know from either \autoref{sum of 3 pencils with a minimal} or Theorem 3.1 in \cite{Voisin92} that $h^0(C,A_1\otimes A_2)\le 4$ holds for all pairs $(A_1,A_2)\in W^1_{a+1}(C)\times W^1_{a+1}(C)$. On the other hand, if $A_1\ne A_2\in W^1_{a+1}(C)$, then the basepoint free pencil trick on the multiplication map 
$$H^0(C,A_1)\otimes H^0(C,A_2)\to H^0(C,A_1\otimes A_2)$$
forces $h^0(C,A_1\otimes A_2)\ge 4$, so in fact $h^0(C,A_1\otimes A_2)=4$. Since $h^0(C,A\otimes A)=3$ for all $A\in W^1_{a+1}(C)$ thanks to the fact that $C$ is (Petri) general, in this situation we have the explicit description
$$h^0(C,A_1\otimes A_2)=\begin{cases}
    3,\textnormal{ if }A_1=A_2\in W^1_{a+1}(C),\\
    4,\textnormal{ if }A_1\ne A_2\in W^1_{a+1}(C).
\end{cases}$$
\vspace{0.2cm}

The first instance when the genus is odd, $g=2a+1$, holds when $r=4$. Here, the only option is $a_1=a_2=a+2$ and the generality of $C$ implies $G^1_{a+2}(C)=W^1_{a+2}(C)$ is a smooth connected curve of genus (see e.g. Theorem 4\cite{EH87})
$$g\left(W^1_{a+2}(C)\right)=1+\frac{a}{a+1}\binom{2a+2}{a}.$$
Thanks to work done in \cite{FO11} and \cite{Ortega13}, it is known that $\grond_{C,a+2,a+2}$ is nonempty of pure dimension $1$ and that it is a symmetric correspondence on $W^1_{a+2}(C)$, whose degree is given by the Castelnuovo number 
$$C(2a+1,a-1,3a-2)=1+\frac{a-2}{a+2}\binom{2a+1}{a}.$$
This situation generalizes naturally to pairs of pencils of the same arbitrary degree, as the subsection 5.4 below shows. 

\subsection{The dimensional upper bound in \autoref{deg locus in sum of pencils}} 

Consider the stratifications 
$$W^1_{a_2}(C)\supseteq W^2_{a_2}(C)\supseteq W^3_{a_2}(C)\supseteq \dots\textnormal{ and }W^1_{a_1}(C)\supseteq W^2_{a_1}(C)\supseteq W^3_{a_1}(C)\supseteq \dots$$
and define $\widetilde{W}^s_{a_2}(C)=W^s_{a_2}(C)\setminus W^{s+1}_{a_2}(C)$ as well as $\widetilde{W}^s_{a_1}(C)=W^s_{a_1}(C)\setminus W^{s+1}_{a_1}(C)$, for all $s\ge 1$. For all $\eps_1,\eps_2\ge 0$, we take a closer look at the sublocus in $\grond_{C,a_1,a_2}$ consisting of pairs $((A_1,V_1),(A_2,V_2))\in G^1_{a_1}(C)\times G^1_{a_2}(C)$ such that 
\begin{align}\label{subcondition on pairs of pencils 2}
(A_1,A_2)\in \widetilde{W}^{1+\eps_1}_{a_1}(C)\times \widetilde{W}^{1+\eps_2}_{a_2}(C)\textnormal{ and }h^0(C,A_1\otimes A_2)\ge r+1.
\end{align}
As of Riemann-Roch, the requirement $h^0(C,A_1\otimes A_2)\ge r+1$ is the same as having an effective divisor $D\in C_{2g-a_1-a_2-2}$ such that 
$$\omega_C=A_1\otimes A_2\otimes \OO_C(D).$$
If we define $B_2=\omega_C\otimes A_2^{-1}\in W^{g-a_2+\eps_2}_{2g-2-a_2}(C)$, we see that looking for pairs $\left((A_1,V_1),(A_2,V_2)\right)\in G^1_{a_1}(C)\times G^1_{a_2}(C)$ such that \eqref{subcondition on pairs of pencils 2} holds is the same as the following package: 

\textbf{(a)} a complete linear series $B_2\in W^{g-a_2+\eps_2}_{2g-2-a_2}(C)$ and a line bundle $\OO_C(D)$ corresponding to an effective divisor $D\in C_{2g-a_1-a_2-2}$ such that $h^0(B_2(-D))=2+\eps_1$, 

\textbf{(b)} a point $V_1$ in the Grassmannian $G(2,H^0(C,B_2(-D)))=G(2,2+\eps_1)$, and 

\textbf{(c)} a point $V_2$ in the Grassmannian $G(2,H^0(C,A_2))=G(2,2+\eps_2)$, where $A_2=\omega_C\otimes B_2^{-1}$. 
\vspace{0.2cm}

Let us first consider only the moduli for \textbf{(a)}. Here, we see that the locus of pairs $(A_1,A_2)\in \widetilde{W}^{1+\eps_1}_{a_1}(C)\times \widetilde{W}^{1+\eps_2}_{a_2}(C)$ such that $h^0(C,A_1\otimes A_2)\ge r+1$ is contained in 
$$\left\{(B_2,\OO_C(D))\mid B_2\in G^{g-a_2+\eps_2}_{2g-2-a_2}(C)\textnormal{ and }D\in V_{2g-a_1-a_2-2}^{g-a_2+\eps_2-\eps_1-1}(B_2)\right\}.$$
As of \autoref{super special secant divisors}, we know this has a dimensional upper bound of 
\begin{align*}
    \rho(g,1+\eps_2,a_2)&+(2g-a_1-a_2-2)-(2+\eps_1)(g-a_1-1+\eps_1-\eps_2)=\\
    &=r-3-\eps_1(g+1-a_1)-\eps_2(g+1-a_2)-\left(\eps_1^2+\eps_2^2-\eps_1\eps_2\right).
\end{align*}
When we also take into account \textbf{(b)} and \textbf{(c)}, we add an extra $\dim G(2,2+\eps_1)=2\eps_1$ and $\dim G(2,2+\eps_2)=2\eps_2$ to the above. Therefore, the sublocus described in \eqref{subcondition on pairs of pencils 2} is of dimension at most 
$$r-3-\eps_1(g-1-a_1)-\eps_2(g-1-a_2)-\left(\eps_1^2+\eps_2^2-\eps_1\eps_2\right).$$
If both $a_1\le g-1$ and $a_2\le g-1$ hold, then the above quantity is automatically bounded above by $r-3$. If $a_2=g$, then we must have $a_1\le g-2$, since $a_1+a_2=d\le 2g-2$ (since $\rho(g,r,d)=2g-d-2$ here). In this case, we want to show 
$$\eps_1(g-1-a_1)+\eps_2(g-1-a_2)+\left(\eps_1^2+\eps_2^2-\eps_1\eps_2\right)=\eps_1(g-1-a_1)-\eps_2+\left(\eps_1^2+\eps_2^2-\eps_1\eps_2\right)\ge 0$$
holds. If $\eps_1\ge \eps_2+1$, this quantity rewrites as 
$$\eps_1(g-a_1-2)+\eps_2(1+\eps_2)+\eps_1(\eps_1-\eps_2-1)+2(\eps_1-\eps_2)\ge 0.$$
If $\eps_1=\eps_2=\eps$, then the inequality becomes 
$$g-a_1-2+\eps^2\ge 0,\textnormal{ again true}.$$
The last option is $\eps_1\le \eps_2-1$, in which case the inequality can be written into 
$$\eps_1(g-a_1-1)+\eps_1^2+\eps_2(\eps_2-\eps_1-1)\ge 0,\textnormal{ true}.$$
All in all, all subloci of $\grond_{C,a_1,a_2}$ as defined in \eqref{subcondition on pairs of pencils 2} are of dimension at most $r-3$. This settles the upper bound $\dim \grond_{C,a_1,a_2}\le r-3$. 

\subsection{Special tensors of pencils of the same degree}

In order to ensure the existence part of \autoref{deg locus in sum of pencils}, we need the following result. A byproduct of this lemma is an  interesting enumerative formula which comes as a consequence of the connection with the varieties of secant divisors. 

\begin{lemma}\label{correspondence on pencils of same degree}
Let $g,r,b\ge 1$ be integers such that $r\ge 4$, $g-2b+r=1$, $\rho(g,r,2b)\ge 0$ and $g-1\ge b\ge \ceiling {g/2}+1$. Let $C$ be a general curve of genus $g$. It is known that $G^1_b(C)\times G^1_b(C)$ is a variety of dimension $2(r-3)$. Then the locus 
$$\grond_{C,b,b}=\left\{\left((A_1,V_1),(A_2,V_2)\right)\in G^1_b(C)\times G^1_b(C)\mid h^0(C,A_1\otimes A_2)\ge r+1\right\}$$
is nonempty and of dimension $r-3$. Furthermore, both projections on the two factors $G^1_b(C)$ are generically finite and realize $\grond_{C,b,b}$ as a symmetric correspondence on $G^1_b(C)$, whose degree is the Castelnuovo number 
$$C(g,g-b,2g-2-b)=\sum_{i=0}^{g-b-1}\frac{(-1)^i}{g-b-i}\binom{g-i-1}{b}\binom{g-i-2}{b-1}\binom{g}{i}.$$
\end{lemma}
We remark that $\grond_{C,b,b}$ is the particular situation $a_1=a_2=b$ in \autoref{deg locus in sum of pencils}. Indeed, observe that in the main theorem we have $g-d+r=1$, which gives $\rho(g,r,d)=2g-2-d\ge 0$. If the degree $d$ is even and $d=2b$, then the condition $g\ge b\ge \ceiling {g/2}+1$ simply reads as $g-1\ge b\ge \ceiling {g/2}+1$. 
\begin{proof}
A generic $A_2\in G^1_b(C)$ satisfies $h^0(C,A_2)=2$, so it is complete. The same construction with $B=\omega_C\otimes A_2^{-1}\in W^{g-b}_{2g-2-b}(C)$ shows that we are interested in effective divisors $D$ of degree $2g-2b-2$ that impose $g-b-1$ conditions on $\abs{B}$. As of Theorem 0.7 in \cite{Farkas08}, there are finitely many such divisors and their number is $C(g,g-b,2g-2-b)$. Every such divisor corresponds uniquely to some $(A_1,V_1)\in G^1_b(C)$ such that $h^0(C,A_1\otimes A_2)\ge r+1$. Since $G^1_b(C)$ has dimension $\rho(g,1,b)=r-3$, we conclude that $\grond_{C,b,b}$ is exactly of dimension $r-3$ and has a generically finite fiber over $G^1_b(C)$ via either of the two projections, both of degree $C(g,g-b,2g-2-b)$. 
\end{proof}

\subsection{The existence part in \autoref{deg locus in sum of pencils} and the component $\Sigma$}

Let $(A_1,V_1)\in G^1_{a_1}(C)$ be a generic pencil. We show that there exists $A_2\in W^1_{a_2}(C)$ such that $h^0(C,A_1\otimes A_2)\ge r+1$. Indeed, we know from \autoref{correspondence on pencils of same degree} that there exist pencils $A_2'\in W^1_{a_1}(C)$ such that $h^0\left(C,A_1\otimes A_2'\right)\ge 2a_1+2-g$. By Riemann-Roch, this is the same as $h^0\left(C,\omega_C\otimes A_1^{-1}\otimes A_2'^{-1}\right)\ge 1$, so there exists an effective divisor $D'\in C_{2g-2a_1-2}$ such that $\omega_C=A_1\otimes A_2'\otimes \OO_C(D')$. Let $p_1,p_2,\dots,p_{a_2-a_1}$ be $a_2-a_1$ points in the support of $D'$. If we write $A_2=A_2'\otimes \OO_C\left(p_1+\dots+p_{a_2-a_1}\right)\in W^1_{a_2}(C)$ and $D=D'-p_1-\dots-p_{a_2-a_1}\in C_{2g-a_1-a_2-2}$, then 
$$\omega_C=A_1\otimes A_2\otimes \OO_C(D)$$
and $D$ is still effective. From this it follows that $h^0\left(C,\omega_C\otimes A_1^{-1}\otimes A_2^{-1}\right)\ge 1$, so $h^0(C,A_1\otimes A_2)\ge r+1$. 

Now let $(A_1,V_1)\in G^1_{a_1}(C)$ be a pencil that is also complete; if $a_1=\ceiling {g/2}+1$ is minimal, then this is automatic. Define $B_1=\omega_C\otimes A_1^{-1}\in W^{g-a_1}_{2g-2-a_1}(C)$, so $h^0(C,B_1)=g-a_1+1$. The argument above shows that there exist secant divisors $D$ of degree $2g-a_1-a_2-2$ that impose at most $g-a_1-1$ conditions on $\abs{B_1}$. Therefore, the variety $V^{g-a_1+1}_{2g-a_1-a_2-2}(\abs{B_1})$ is nonempty and its determinantal structure tells us it is of dimension at least $a_2-a_1$. From Corollary 0.3 in \cite{Farkas08}, we deduce that if $B_1\in W^{g-a_1}_{2g-2-a_1}(C)$ is general i.e. if $(A_1,V_1)\in G^1_{a_1}(C)$ is general, then $V^{g-a_1+1}_{2g-a_1-a_2-2}(\abs{B_1})$ is precisely of dimension $a_2-a_1$ and its generic point is a divisor that imposes precisely $g-a_1-1$ conditions on the linear series $\abs{B_1}$. In other words, there is an $(a_2-a_1)$-dimensional family of complete pencils $(A_2,V_2)\in G^1_{a_2}(C)$ such that $h^0(C,A_1\otimes A_2)\ge r+1$, granted $(A_1,V_1)\in G^1_{a_1}(C)$ is general. 
\vspace{0.2cm}

If $\rho(g,1,a_1)=0$ i.e. $g=2a$ is even and $a_1=a+1$, then $G^1_{a+1}(C)=W^1_{a+1}(C)$ is reduced and finite and consists of $w^1_{a+1}=\frac{1}{a+1}\binom{2a}{a}$ points. The fiber of the map $\grond_{C,a+1,a_2}\to W^1_{a+1}(C)$ over each point $A_1\in W^1_{a+1}(C)$ is at most $r-3$ dimensional, according to subsection 5.3. On the other hand, we just argued that the dimension of each such fiber is at least $a_2-a_1$ and, in this case, we have $a_2-a_1=r-3$. Consequently, over each point $A_1\in W^1_{a+1}(C)$ there is an $(r-3)$-dimensional fiber. 

If $\rho(g,1,a_1)\ge 1$, then $G^1_{a_1}(C)$ is irreducile. The discussion from two paragraphs ago ensures the existence of a component $\Sigma$ of $\grond_{C,a_1,a_2}$ that maps dominantly onto $G^1_{a_1}(C)$ and which has dimension 
$$\dim \Sigma=\rho(g,1,a_1)+(a_2-a_1)=a_1+a_2-g-2=r-3.$$
In order to prove the equality $h^0(C,A_1\otimes A_2)=r+1$ for the generic point $\left((A_1,V_1),(A_2,V_2)\right)$ of $\Sigma$, by semicontinuity it suffices to prove this holds for a single point. But if $(A_1,V_1)\in G^1_{a_1}(C)$ is general, then it is complete and we know from \autoref{no quads of rk 4 generically} that the multiplication map 
$$\Sym^2 H^0\left(C,\omega_C\otimes A_1^{-1}\right)\to H^0\left(C,\omega_C^{2}\otimes A_1^{-2}\right)$$
does not contain quadrics of rank at most $4$ in its kernel. This means that we cannot write $\omega_C=A_1\otimes A_2\otimes A_3$ for any two other pencils $A_2$ and $A_3$, so $h^0\left(C,\omega_C\otimes A_1^{-1}\otimes A_2^{-1}\right)\le 1$ i.e. $h^0(C,A_1\otimes A_2)\le r+1$ must hold for all $A_2\in W^1_{a_2}(C)$. Since $h^0(C,A_1\otimes A_2)\ge r+1$ is guaranteed in the fiber over $(A_1,V_1)$, we conclude that $h^0(C,A_1\otimes A_2)=r+1$ for all $\left((A_1,V_1),(A_2,V_2)\right)$ in the fiber of $\Sigma$ over $(A_1,V_1)$. 

Finally, the same argument works when $\rho(g,1,a_1)=0$ i.e. $g=2a$ is even and $a_1=a+1$, with the further remark that the bound $h^0(C,A_1\otimes A_2)\le r+1$ holds for all points $(A_1,A_2)\in W^1_{a+1}(C)\times W^1_{a_2}(C)$, thanks to \autoref{no quads of rk 4 if minimal pencil}. We conclude that, here, $h^0(C,A_1\otimes A_2)=r+1$ is true for all $\left((A_1,V_1),(A_2,V_2)\right)\in \grond_{C,a+1,a_2}$. 

\subsection{The incidence variety of linear series and rank $4$ quadrics}

Now that we have the entire proof of \autoref{deg locus in sum of pencils}, we are in a position to prove the following

\begin{theorem}\label{incidence variety for rank 4 quadrics}
Let $g,r,d\ge 1$ be integers such that $r\ge 4$, $s=d-g-r\ge -1$ and $\rho(g,r,d)\ge 0$. Let $C$ be a general curve of genus $g$. Consider the incidence variety 
$$\irond_{C}^4=\left\{(\ell,q)\mid \ell\in G^r_d(C)\textnormal{ and }q\in \Sigma_4(C,\ell)\right\}.$$
Then 

$\bullet$ for $r=4$, we have $\dim \irond_C^4=3s+4$,  

$\bullet$ for $r\ge 5$ and $s\ge 0$, $\irond_C^4$ has a component of dimension $(r-1)(s+2)-2$, and 

$\bullet$ for $r\ge 5$, $s=-1$ and $\ceiling {g/2}\ge r-2$, $\irond_C^4$ has a component of dimension $r-3$. 
\end{theorem}
\begin{proof}
We use the alternative description of $\irond_C^4$ given by \autoref{character of rank 3 and 4}. By including the base locus $F$ in either of the two pencils $(A_1,V_1)\in G^1_{a_1}(C)$ and $(A_2,V_2)\in G^1_{a_2}(C)$ if necessary, it is enough to prove that, as we go through any splitting $d=a_1+a_2$ with $\ceiling {g/2}+1\le a_1\le a_2$, the varieties 
$$\irond_{C,a_1,a_2}^4=\left\{\left((A_1,V_1),(A_2,V_2),V\right)\mid \textnormal{ condition }(*)\textnormal{ holds}\right\}$$
satisfy the statement. Here, condition $(*)$ is the following two requirements: 

\textbf{(1)} $(A_1,V_1)\in G^1_{a_1}(C)$ and $(A_2,V_2)\in G^1_{a_2}(C)$, and

\textbf{(2)} $V\in G(r+1,H^0(C,A_1\otimes A_2))$ contains the image of $V_1\otimes V_2\to H^0(C,A_1\otimes A_2)$. 
\vspace{0.2cm}
\\
Since the sublocus in $\irond_{C,a_1,a_2}^4$ where $\rk(q)=3$ is already proven to be of dimension at most $(r-1)(s+1)-1<(r-1)(s+2)-2$, we may ignore it from now on. 
\vspace{0.2cm}

Suppose first that $s=d-g-r\ge 0$. Over the stratum in $G^1_{a_1}(C)\times G^1_{a_2}(C)$ where $h^0(C,A_1\otimes A_2)=d+1-g$ is minimum possible, the fibers of $\irond_{C,a_1,a_2}^4$ parametrize the vector spaces $V$ that satisfy condition \textbf{(2)} from above. The fiber over $((A_1,V_1),(A_2,V_2))$ is the Grassmannian $G(r-3,H^0(C,A_1\otimes A_2)/V_1\otimes V_2)$, so we obtain a component of $\irond_{C,a_1,a_2}^4$ that is of dimension 
$$\underbrace{\rho(g,1,a_1)+\rho(g,1,a_2)}_{\textnormal{dimension of }G^1_{a_1}(C)\times G^1_{a_2}(C)}+\underbrace{(r-3)(d-g-r)}_{\textnormal{dimension of Grassmannian}}=(r-1)(s+2)-2.$$
When $s=d-g-r=-1$, the image of the map $\irond_{C,a_1,a_2}^4\to G^1_{a_1}(C)\times G^1_{a_2}(C)$ is the subvariety $\grond_{C,a_1,a_2}$ that was studied in \autoref{deg locus in sum of pencils}. From the same theorem, as $\ceiling {g/2}+1\le a_1\le a_2\le g$ (which is the case, since it is assumed that $\ceiling{g/2}\ge r-2$), we infer the existence of a component $\Sigma$ of dimension $r-3$ in $\grond_{C,a_1,a_2}$ where the relation $h^0(C,A_1\otimes A_2)=d+2-g=r+1$ holds generically. For a generic point $((A_1,V_1),(A_2,V_2))\in \Sigma$, condition \textbf{(2)} holds if and only if $V=H^0(C,A_1\otimes A_2)$, so $V$ is determined uniquely. This gives the desired component of dimension $r-3$ in $\irond_{C,a_1,a_2}^4$. 
\vspace{0.2cm}

For $r=4$, we are able to set an upper bound on the extra strata. Consider again the situation $s=d-g-4\ge 0$ first. For all integers $h\ge 1$ and when $\ceiling{g/2}+1\le a_1\le a_2\le g$, combining \autoref{super special secant divisors} and the argument provided in subsection 5.3, we see that 
$$\grond_{C,a_1,a_2}^h:=\left\{\left((A_1,V_1),(A_2,V_2)\right)\in G^1_{a_1}(C)\times G^1_{a_2}(C)\mid h^0(C,A_1\otimes A_2)\ge d+1-g+h\right\}$$
is of dimension at most $d-g-1-h$; notice that $\grond_{C,a_1,a_2}^1=\grond_{C,a_1,a_2}$. Adding the moduli of the vector space $V\in G(5,H^0(C,A_1\otimes A_2))$ that needs to contain the image of $V_1\otimes V_2\to H^0(C,A_1\otimes A_2)$, it follows that, for each $h\ge 1$, the dimension of $\irond_{C,a_1,a_2}^4$ over the stratum $\grond_{C,a_1,a_2}^h\setminus \grond_{C,a_1,a_2}^{h+1}$ is at most 
$$d-g-1-h+\underbrace{(5-4)(d+h-g-4)}_{\textnormal{dimension of }G(5-4,H^0(C,A_1\otimes A_2)/V_1\otimes V_2)}=2s+3<3s+4.$$
When $a_2\ge g+1$, we are in the area $d=a_1+a_2\ge \ceiling{3g/2}+2$. Here, the residual of any series of degree $d$ is of degree $2g-2-d\le \floor {g/2}-4$, so it has at most one section. In other words, the only possible extra stratum is $\grond_{C,a_1,a_2}^1$ and it is trivially of codimension at least $1$ in $G^1_{a_1}(C)\times G^1_{a_2}(C)$. The dimension of $\irond_{C,a_1,a_2}^4$ over this extra stratum is thus at most 
$$\rho(g,1,a_1)+\rho(g,1,a_2)-1+\underbrace{(5-4)(d+1-g-4)}_{\textnormal{dimension of }G(5-4,H^0(C,A_1\otimes A_2)/V_1\otimes V_2)}=3s+4.$$
Finally, when $s=d-g-4=-1$ i.e. $d=g+3$, the same argument involving \autoref{super special secant divisors} shows that there is only one possible extra stratum, namely $\grond_{C,a_1,a_2}^2$, and furthermore $\dim \grond_{C,a_1,a_2}^2\le 0$ i.e. it contains at most finitely many points. Over every such hypothetical point $\left((A_1,V_1),(A_2,V_2)\right)$, given that $h^0(C,A_1\otimes A_2)=6$, the vector spaces $V\le H^0(C,A_1\otimes A_2)$ containing $V_1\otimes V_2$ vary in the Grassmannian $G(1,2)\cong \PP^1$. This establishes the upper bound $\dim \irond_{C,a_1,a_2}^4\le 1$ and finishes the proof of the lemma. 
\end{proof}

\subsection{Finiteness of the multiplication map of pencils}

In this subsection we finish the proof of \autoref{SSMRC rank 4}. Below, we specify explicitly the lower bounds that were left imprecise in the introduction. Recall that the space of quadrics in the projective space $\PP^r$ is parametrized by $\PP^N$, where $N=\binom{r+2}{2}-1$. In $\PP^N$, the degree of the variety $\Sigma_4(\PP^r)$ parametrizing quadrics of rank at most $4$ in $\PP^r$ is given by the formula (see e.g. \cite{HT84})
$$\delta_r:=\deg \Sigma_4\left(\PP^r\right)=\frac{1}{r(r-1)}\binom{2r-4}{r-2}\binom{2r-3}{r-1},\textnormal{ for all }r\ge 4.$$

\begin{theorem}
The following table describes the knowledge regarding the degeneracy locus $Q^r_{d,4}(C)$ and its dimension, as it is predicted by \autoref{SSMRC} when $s=d-g-r\ge -1$.

\[
\begin{tabular}{|c|c|c|}
\hline
$k=4$ & $r=4$ & $r\ge 5$ \\
\hline
$s=-1$ & \makecell{empty for $g=5$ and $g=6$, \\ dimension $1$ for $g\ge 7$, as expected} & \makecell{a component of expected dimension $r-3$ \\ when $a\ge 3r-11$ for $g=2a$ even and \\ $\frac{a}{a+1}\binom{2a+2}{a}\ge \binom{\delta_r-1}{2}$ for $g=2a+1$ odd} \\
\hline
$s\ge 0$ & \makecell{expected dimension $3s+4$ \\ for $g\ge s+4$} & \makecell{a component of expected dimension \\ $(r-1)(s+2)-2$ \\ when $g\ge 2r+s-4$ is even and \\ $\frac{a}{a+1}\binom{2a+2}{a}\ge \binom{\delta_{r+s}-1}{2}$ when $g=2a+1$ odd} \\
\hline
\end{tabular}
\]
\end{theorem}

For small values of $r\ge 5$, the lower bounds on the genus from above can be made explicit. Here is the evolution of the bound when $s=-1$:
\[
\begin{tabular}{|c|c|c|c|c|c|}
\hline
$r$ & $5$ & $6$ & $7$ & $8$ & $9$ \\
\hline
$g$ & $8\textnormal{ and }\ge 10$ & $14,16\textnormal{ and }\ge 18$ & $20,22\textnormal{ and }\ge 24$ & $26,28\textnormal{ and }\ge 30$ & $32,34,36\textnormal{ and }\ge 38$  \\
\hline
\end{tabular}
\]

We treat separately the values $(g,r,d)\in\{(5,4,8),(6,4,9)\}$, as they are the only instances in the above where $q(g,r,d,4)\ge 0$. 

The situation $(g,r,d)=(5,4,8)$ concerns the general canonical curve $C$ of genus $5$. The canonical embedding of a general $[C]\in\mrond_5$ sits on a $1$-dimensional family of rank $4$ quadrics (see \cite{Kadikoylu19}). This family of quadrics induces an involution of the Brill-Noether curve $W^1_4(C)$, which is 
$$\iota:W^1_4(C)\to W^1_4(C)\textnormal{ defined by }\iota(A)=\omega_C\otimes A^{-1}.$$

The other situation is $(g,r,d)=(6,4,9)$. Here $W^4_9(C)\cong C$ via the isomorphism $\omega_C(-p)\mapsfrom p$, and all $g^4_9$'s are complete and very ample. Furthermore, the multiplication map $\nu(L)$ is surjective for every $L\in G^4_9(C)$ thanks to \eqref{GL degree bound} and $\dim \ker\nu(L)=2$. This means that the embedded curve $C\stackrel{\abs{L}}{\injto}\PP^4$ lies on a pencil of quadrics. Any of these quadrics that is of rank $4$ must correspond to a splitting of $L$ as a sum of two pencils, one of which is of minimal of degree $4$. Since there are precisely $5=w^1_4$ such minimal pencils, we deduce the pencil of quadrics is transversal to $\Sigma_4(\PP^4)$ and they intersect in precisely $5=\deg \Sigma_4(\PP^4)$ points, each of them corresponding to one of the $5$ minimal pencils on $C$. Consequently, there is no special behavior and $Q^4_{9,4}(C)=\emptyset$ for the general $[C]\in\mrond_6$.  
\vspace{0.2cm}

From now on, our focus is on values for which $q(g,r,d,4)=-1$. Then the forgetful map $\pi:\irond_C^4\to Q^r_{d,4}(C)$ is well defined, as the special behavior of $Q^r_{d,4}(C)$ in $G^r_d(C)$ is given by linear series $\ell$ for which $\ker\nu(\ell)$ contains quadrics of rank at most $4$. Since $\pi$ is obviously surjective, we immediately obtain $\dim Q^r_{d,4}(C)\le \dim \irond_C^4=\beta(g,r,d,4)$. 

As we did in the proof of \autoref{SSMRC rank 3}, we show that the map $\pi$ is generically finite from at least one irreducible component of the expected dimension in $\irond_C^4$. With the use of \autoref{incidence variety for rank 4 quadrics}, this settles the equality $\dim Q^r_{d,4}(C)=\beta(g,r,d,4)$ (or just the existence of a component of the expected dimension $\beta(g,r,d,4)$, when $r\ge 5$) and finishes the proof of \autoref{SSMRC rank 4}. 

It suffices to look at the varieties $\irond_{C,a_1,a_2}^4$ that were introduced in the previous section. In the spirit of the previous paragraph, we focus our attention to the situation when $a_1=\ceiling{g/2}+1$, so at least one of the two pencils that give a ruling for the rank $4$ quadric is of minimal degree. We first deal with the case of special series, when $s=d-g-r=-1$. 

\begin{lemma}\label{finite projection}
Let $g,r,d$ be positive integers such that $d=g+r-1$, $\rho(g,r,d)\ge 0$, $r\ge 4$ and let $\ceiling {g/2}+1=a_1\le a_2\le g$ be such that $a_1+a_2=d$. If $C$ is a general curve of genus $g$, then the map 
\begin{align*}
    \mu:\grond_{C,a_1,a_2}\to W^r_{d}(C)\textnormal{ given by }
    \left((A_1,V_1),(A_2,V_2)\right)\mapsto A_1\otimes A_2
\end{align*}
is finite if 

$\bullet$ $g=2a$ is even and $a\ge 3r-11$, 

$\bullet$ $g=2a+1$ is odd and $a\ge 3$ for $r=4$, or if

$\bullet$ $g=2a+1$ is odd and $\frac{a}{a+1}\binom{2a+2}{a}\ge\binom{\delta_r-1}{2}$ for $r\ge 5$. 
\end{lemma}

\vspace{0.2cm}

The lemma implies the special range $s=-1$ when the genus $g$ is large enough. Indeed, not only that $\grond_{C,a_1,a_2}\to W^r_d(C)$ is a finite map, but we also infer from \autoref{deg locus in sum of pencils} that there are points in $\grond_{C,a_1,a_2}$ such that $h^0(C,A_1\otimes A_2)=r+1$, so $V=H^0(C,A_1\otimes A_2)$ is uniquely determined. Consequently, $\grond_{C,a_1,a_2}$ and $\irond_{C,a_1,a_2}^4$ are birationally equivalent and the commutative diagram below
\[\begin{tikzcd}
	{\irond_{C,a_1,a_2}^4} & {Q^r_{d,4}(C)} \\
	{\grond_{C,a_1,a_2}} & {W^r_d(C)}
	\arrow["{\textnormal{forget }q}", from=1-1, to=1-2]
	\arrow["{\textnormal{forget } V}"', two heads, from=1-1, to=2-1]
	\arrow["{\textnormal{forget } V}", from=1-2, to=2-2]
	\arrow["\mu", from=2-1, to=2-2]
\end{tikzcd}\]
implies $Q^r_{d,4}(C)$ contains a component of the expected dimension. 

\vspace{0.2cm}

\begin{proof}
[Proof of \autoref{finite projection}]
\textit{\underline{Step I}}: We deal with the even genus $g=2a$ case first. Here, $a_1=a+1$ and the condition $a\ge 3r-11$ translates into $\rho(g,2,a_2)<0$. Consequently, any point in $\grond_{C,a_1,a_2}$ corresponds to two complete linear series $A_1\in G^1_{a+1}(C)$ and $A_2\in G^1_{a_2}(C)$. From the description of $\grond_{C,a_1,a_2}$ given by \autoref{deg locus in sum of pencils}, we know that to each of the finitely many points $A_1$ in $G^1_{a+1}(C)$ corresponds an $(r-3)$-dimensional family $\frond(A_1)$ of linear series $A_2\in G^1_{a_2}(C)$ such that $h^0(C,A_1\otimes A_2)=r+1$, all of which we just argued they must be complete. Fixing one such $A_1$, the map $\frond(A_1)\to W^r_d(C)$ given by $A_2\mapsto A_1\otimes A_2$ is injective, so its image is also $(r-3)$-dimensional; this holds for all $A_1\in G^1_{a+1}(C)$. The even genus case has been proved. 
\vspace{0.2cm}

\textit{\underline{Step II}}: We continue with the odd genus $g=2a+1$ situation, when $r\ge 4$. Here, $a_1=a+2$, $a_2=a+r-2$ and, if $r\ge 5$, we have the extra condition $\frac{a}{a+1}\binom{2a+2}{a}\ge \binom{\delta_r-1}{2}$. All points in $G^1_{a+2}(C)$ are complete linear series. 

Assume that the map $\grond_{C,a+2,a+r-2}\to W^r_{2a+r}(C)$ is not finite. Then there is a line bundle $L\in W^r_{2a+r}(C)$ that can be written as $L=A\otimes B$ for infinitely many pairs $(A,B)\in \grond_{C,a+2,a+r-2}$. Fixing one such pair $(A_0,B_0)$, it follows that the embeddings 
\begin{align*}
    W^1_{a+2}(C)&\injto \textnormal{Jac}(C),\textnormal{ given by }A\mapsto A\otimes A_0^{-1}\textnormal{ and}\\
    W^1_{a+r-2}(C)&\injto \textnormal{Jac}(C),\textnormal{ given by }B\mapsto B_0\otimes B^{-1}
\end{align*}
intersect in infinitely many points. Now, we separate the argument in two parts.
\vspace{0.2cm}

$\bullet$ \textit{Case I}: $r=4$ and $a\ge 3$. The above can happen only if the embedded images of $W^1_{a+2}(C)$ coincide. In other words, the line bundle $L\in W^4_{2a+4}(C)$ gives an involution 
$$\iota:W^1_{a+2}(C)\to W^1_{a+2}(C)\textnormal{ defined by }\iota(A)=L\otimes A^{-1}.$$
We remark two things. First, this implies $L$ can be written $L=A\otimes B$ such that $A\in W^1_{a+2}(C)$ is generic, so $h^0(C,L)=5$ holds, thanks to the description of $\grond_{C,a+2,a+2}$ given by \autoref{deg locus in sum of pencils}. Second, $\iota$ cannot have fixed points, since the curve $C$ is Petri general. Therefore, the quotient $W^1_{a+2}(C)/\iota$ is a smooth curve that parametrizes the rank $4$ quadrics in $\ker\nu(L)$, where $\nu(L)$ is the multiplication map 
$$\nu(L):\underbrace{\Sym^2 H^0(C,L)}_{\textnormal{dimension 15}}\to \underbrace{H^0\left(C,L^2\right)}_{\textnormal{dimension }2a+8}.$$
This curve $W^1_{a+2}(C)/\iota$ is the intersection of $\PP\ker\nu(L)$ with $\Sigma_4(\PP^4)$ in $\PP \Sym^2 H^0(C,L)\cong \PP^{14}$. Since $\Sigma_4(\PP^4)$ is of codimension $1$ in $\PP^{14}$, the only way in which this could happen is if $\PP\ker\nu(L)\cong \PP^{\ell}$ is such that $\ell\in \{1,2\}$. Now, the genus of $W^1_{a+2}(C)/\iota$ is given by 
$$g\left(W^1_{a+2}(C)/\iota\right)=1+\frac{1}{2}\left(g\left(W^1_{a+2}(C)\right)-1\right)=1+\frac{a}{2a+2}\binom{2a+2}{a},$$
so $W^1_{a+2}(C)/\iota$ is not rational and it follows that $\ell=1$ cannot hold. If $\ell=2$, as $\deg Q_4(\PP^4)=5$, we deduce that $W^1_{a+2}(C)/\iota$ is a smooth curve of degree $5$ in $\PP\ker\nu(L)\cong \PP^2$. Consequently, the genus of $W^1_{a+2}(C)/\iota$ is $\binom{5-1}{2}=6$, which could only hold if $a=2$. This is a contradiction. 
\vspace{0.2cm}

$\bullet$ \textit{Case II}: $r\ge 5$ and $\frac{a}{a+1}\binom{2a+2}{a}\ge \binom{\delta_r-1}{2}$. The condition means that the embedded image of $W^1_{a+2}(C)$ is fully contained in the embedded image of $W^1_{a+r-2}(C)$. In other words, the line bundle $L\in W^r_{2a+r}(C)$ gives an embedding 
$$\iota:W^1_{a+2}(C)\injto W^1_{a+r-2}(C)\textnormal{ defined by }\iota(A)=L\otimes A^{-1}.$$
As in the previous case, this implies $h^0(C,L)=r+1$. On the other hand, the embedding $\iota$ gives here only a component of the locus that parametrizes rank $4$ quadrics in $\ker\nu(L)$, where $\nu(L)$ is the multiplication map 
$$\nu(L):\underbrace{\Sym^2 H^0(C,L)}_{\textnormal{dimension }\binom{r+2}{2}}\to \underbrace{H^0\left(C,L^2\right)}_{\textnormal{dimension }2a+2r}.$$
This component is a curve isomorphic to $W^1_{a+2}(C)$ and it parametrizes rank $4$ quadrics that have a ruling of minimal degree $a+2$. In particular, the intersection $\PP\ker\nu(L)\cap \Sigma_4(\PP^r)$ must have a $1$-dimensional component, so $\PP\ker \nu(L)\cong \PP^\ell$ for some $1\le \ell\le \binom{r-2}{2}+1$, since $\Sigma_4(\PP^r)$ itself is of codimension $\binom{r-2}{2}$ in $\PP \Sym^2 H^0(C,L)$. If $\ell'\le \ell$ is the smallest dimension of a hyperplane in $\PP^\ell$ that contains the isomorphic image of the curve $W^1_{a+2}(C)$, then this curve is nondegenerate in $\PP^{\ell'}$ and it is clear that $\ell'\ne 1$. Furthermore, the degree of the curve is not greater than $\delta_r=\deg \Sigma_4(\PP^r)$. Using Castelnuovo's bound on nondegenerate curves in a fixed projective space having a fixed degree, we see that the highest bound on the genus is given in the projective plane, so we obtain the estimate 
$$1+\frac{a}{a+1}\binom{2a+2}{a}=g\left(W^1_{a+2}(C)\right)\le \binom{\delta_r-1}{2}.$$
This is a contradiction with our assumption. The conclusion follows. 
\end{proof}

We now also present a similar statement for nonspecial series, when $s=d-g-r\ge 0$. 

\begin{lemma}\label{finite projection 2}
Let $g,r,d$ be positive integers such that $d=g+r+s$ with $s\ge 0$, $r\ge 4$ and let $\ceiling{g/2}+1=a_1\le a_2$ be such that $a_1+a_2=d$. If $C$ is a general curve of genus $g$, then the map 
$$\nu:G^1_{a_1}(C)\times G^1_{a_2}(C)\to W^{r+s}_d(C)\textnormal{ given by }\left((A_1,V_1),(A_2,V_2)\right)\mapsto A_1\otimes A_2$$
is finite if 

$\bullet$ $g=2a$ is even and $g\ge 2r+s-4$, or

$\bullet$ $g=2a+1$ is odd and $\frac{a}{a+1}\binom{2a+2}{a}\ge \binom{\delta_{r+s}-1}{2}$. 
\end{lemma}

Once this is established, then the nonspecial range from \autoref{SSMRC rank 4} follows as we described earlier in this subsection, with the following remark. Since the generic point of $G^1_{a_1}(C)\times G^1_{a_2}(C)$ satisfies $h^0(C,A_1\otimes A_2)=r+s+1$ and the same is true for $W^{r+s}_d(C)$, the commutative diagram 
\[\begin{tikzcd}
	{\irond_{C,a_1,a_2}^4} & {Q^r_{d,4}(C)} \\
	{G^1_{a_1}(C)\times G^1_{a_2}(C)} & {W^{r+s}_{d}(C)}
	\arrow["{\textnormal{forget }q}", from=1-1, to=1-2]
	\arrow["{\textnormal{forget }V}"', from=1-1, to=2-1]
	\arrow["{\textnormal{forget }V}", from=1-2, to=2-2]
	\arrow["\mu", from=2-1, to=2-2]
\end{tikzcd}\]
implies $Q^r_{d,4}(C)$ contains a component of the expected dimension $\beta(g,r,d,4)$. 

\begin{proof}
[Proof of \autoref{finite projection 2}]
\underline{\textit{Step I}}: We present again the even genus $g=2a$ first. Here, $a_1=a+1$ and $a_2=a+r+s-1$. Let $(A_1,V_1)\in G^1_{a+1}(C)$ be any minimal pencil, which in particular is complete. Then pick a generic point $(A_2,V_2)$ in $G^1_{a_2}(C)$ and set $L=A_1\otimes A_2$. If $V\in G(r+1,H^0(C,L))$ is generic with respect to the condition that it contains the image of $V_1\otimes V_2$, we claim $(L,V)$ cannot come from $(A_1,V_1)\in G^1_{a+1}(C)$ with any other pair $(A_2',V_2')\in G^1_{a_2}(C)$. 

First, observe that $h^0(C,L)=d+1-g$ as an immediate consequence of the generality of $(A_2,V_2)\in G^1_{a_2}(C)$. Second, if $(L,V)$ were to come from $(A_1,V_1)$ with any other pair $(A_2',V_2')\in G^1_{a_2}(C)$, then $L=A_1\otimes A_2'$ would be needed, and this forces $A_2=A_2'$. 
\vspace{0.2cm}

Now, if $a_2\le g+1$, then the generality of $(A_2,V_2)\in G^1_{a_2}(C)$ gives $V_2=H^0(C,A_2)$, so $V_2$ is uniquely determined and we are done. 

Else, $a_2\ge g+2$ and the same generality assumption $(A_2,V_2)\in G^1_{a_2}(C)$ implies $h^0(C,A_2)=a_2+1-g$. To prevent $(L,V)$ to come from $(A_1,V_1)$ with any other pair $(A_2,V_2')\in G^1_{a_2}(C)$, it suffices to ensure 
$$V\cap \left(V_1\otimes H^0(C,A_2)\right)=V_1\otimes V_2,$$
as vector subspaces of $H^0(C,L)$. But the bound $g\ge 2r+s-4$ implies $(r+1)+2(a_2+1-g)-(d+1-g)\le 4$, so the generic $V\in G(r+1,H^0(C,L))$ containing $V_1\otimes V_2$ contains no other vector spaces of the form $V_1\otimes V_2'$, for any $V_2'\in G(2,H^0(C,A_2))$. 

\underline{\textit{Step II}}: The odd genus $g=2a+1$ case follows the exact same argument as in \textit{Step II, Case II} from the proof of \autoref{finite projection} and hence we omit it. 
\end{proof}

\section{A Component that Dominates the Jacobian}

In this section, we prove \autoref{SMRC r ge 5}. In the hypothesis of this theorem, the expected dimension of $Q^r_d(C)$ can be computed to 
$$\alpha(g,r,d)=\rho(g,r,d)-1-(2d+1-g)+\binom{r+2}{2}=q(g,d-g,d,r+1)+g,$$
where we recall that $q(g,d-g,d,r+1)$ is the expected dimension of the variety $\Sigma_{r+1}(C,\ell)$ of quadrics of rank at most $r+1$ containing the image of a general curve $C$ via a linear series $\ell\in G^{d-g}_d(C)$. We also recall that in \cite{Kadikoylu19} it is shown that, if the linear series $\ell\in G^{d-g}_d(C)$ is further assumed to be general, then the expected dimension $q(g,d-g,d,r+1)$ is the actual dimension of $\Sigma_{r+1}(C,\ell)$. Using these observations we show, under the assumptions $d>g+r$ and $\binom{r+2}{2}\le 2d+1-g$, that the variety
$$Q^r_d(C)=\{\ell\in G^r_d(C)\mid \nu(\ell)\textnormal{ is not of maximal rank}\}$$
contains a component of the expected dimension $\alpha(g,r,d)$ that maps dominantly onto $\Pic^d(C)$. We know that all irreducible components of $Q^r_d(C)$ have dimension at least $\alpha(g,r,d)$, thanks to its determinantal description. Thus, it suffices to construct a dominant component having dimension at most $\alpha(g,r,d)$. 

\begin{proof}
[Proof of \autoref{SMRC r ge 5}]
For $L\in \Pic^d(C)$, and $3\le k\le r+1$, let $Q^{r,k}(C,L)$ denote the locus in $G(r+1,H^0(C,L))$ of those $V$ for which $\ker\nu(L,V)$ contains a quadric of rank $k$. If $L\in \Pic^d(C)$ is general, we show that $\dim Q^{r,k}(C,L)\le q(g,d-g,d,r+1)$ for all $3\le k\le r+1$. This establishes the existence of a component of $Q^r_{d}(C)$ of dimension at most $q(g,d-g,d,r+1)+\dim \Pic^d(C)=\alpha(g,r,d)$ that dominates $\Pic^d(C)$, as required. 

So let $L\in \Pic^d(C)$ be general. Then $L$ is nonspecial, very ample and embeds $\phi_L:C\injto \PP^{d-g}$, and the multiplication map 
$$\nu(L):\Sym ^2H^0(C,L)\to H^0\left(C,L^2\right)$$
is surjective. Also, \cite{Kadikoylu19} tells us that $\Sigma_k(C,L)$ is of the expected dimension $q(g,d-g,d,k)$, for all $3\le k\le d-g+1$. In particular, as long as $q(g,d-g,d,k)\ge 0$, then $\widetilde{\Sigma}_k(C,L)=\Sigma_k(C,L)\setminus \Sigma_{k-1}(C,L)$ is nonempty and of dimension $q(g,d-g,d,k)$. We adopt the notation $\widetilde{\Sigma}_k(C,\ell)$ for any linear series $\ell$ on $C$ in a similar manner. 
\vspace{0.2cm}

Let $k'$ be the smallest $k\in\{3,4,\dots,r+1\}$ for which $\alpha(g,k-1,d)-g=q(g,d-g,d,k)\ge 0$. Note that $k'$ exists, since $\alpha(g,r,d)-g\ge 0$. For every $k\in\{k',k'+1,\dots,r+1\}$, consider the incidence variety 
\[\begin{tikzcd}
	& {\mathcal{I}_{L}^{r,k}=\left\{(V,q)\mid V\in Q^{r,k}(C,L)\textnormal{ and }q\in \widetilde{\Sigma}_k(C,(L,V))\right\}} & \\
	{Q^{r,k}(C,L)} && {\widetilde{\Sigma}_k(C,L)}
	\arrow["{\pi_1}", from=1-2, to=2-1]
	\arrow["{\pi_2}"', from=1-2, to=2-3]
\end{tikzcd}\]
The second projection makes sense in the following manner. For any point $(V,q)\in\irond_{L}^{r,k}$, the quadric $q\in \ker\nu(L,V)$ is of rank $k$ and can be naturally seen as a quadric $q\in \ker\nu(L)$ that is of the same rank $k$. Indeed, one only needs to consider the commutative diagram 
\[\begin{tikzcd}
	{\ker\nu(L,V)} & {\ker\nu(L)} \\
	{\Sym^2 V} & {\Sym^2 H^0(C,L)}
	\arrow[hook, from=1-1, to=1-2]
	\arrow[hook, from=1-1, to=2-1]
	\arrow[hook, from=1-2, to=2-2]
	\arrow[hook, from=2-1, to=2-2]
\end{tikzcd}\]
Now let $q\in \widetilde{\Sigma}_k(C,L)$ be any quadric of rank $k$ in $\ker\nu(L)$. Since $\rk(q)=k$, one can find linearly independent sections $s_1,s_2,\dots,s_k\in H^0(C,L)$ such that $q$ is in its normal form 
$$q=s_1\otimes s_1+s_2\otimes s_2+\dots +s_k\otimes s_k\in \Sym^2 H^0(C,L).$$ 
For $q$ to be image of a point $(V,q)$ in $\irond_L^{r,k}$, it is needed that $V\in Q^{r,k}(C,L)$ contains the vector space $W_q$ spanned by $s_1,s_2,\dots,s_k$. After this requirement is met, the remaining $r+1-k$ sections of $V$ can be chosen freely. In other words, the fiber over $q\in \widetilde{\Sigma}_k(C,L)$ is the Grassmannian $G(r+1-k,H^0(C,L)/W_q)\cong G(r+1-k,d-g+1-k)$. It follows that 
\begin{align*}
    \dim \irond_L^{r,k}&=\dim \widetilde{\Sigma}_k(C,L)+\dim G(r+1-k,d-g+1-k)=\\
    &=q(g,d-g,d,k)+(r+1-k)(d-g-r)\le q(g,d-g,d,r+1).
\end{align*}
The projection $\pi_1$ sets the immediate upper bound $\dim Q^{r,k}(C,L)\le \dim \irond_L^{r,k}\le q(g,d-g,d,r+1)$, as required. This finishes the proof of the theorem. 
\end{proof}

\section{The Strong Maximal Rank Conjecture in $\PP^4$}

In this section, we prove \autoref{SMRC r=4}. Let $C$ be a general curve of genus $g\ge 1$ and let $\ell=(L,V)$ be a $g^4_d$ on $C$, with $d\ge g+3$. Consider the associated multiplication map $\nu(\ell):\Sym^2 V\to H^0\left(C,L^2\right)$. Under the generality hypothesis $\rho(g,4,d)\ge 0$, one can easily check that the only instances when $s=d-g-4\in\{0,-1\}$ and $\dim_\CC\Sym^2 V=15$ is greater than $2d+1-g=\dim_\CC H^0\left(C,L^2\right)$ are
$$(g,d)\in\underbrace{\{(1,5),(2,6),(3,7),(4,8),(5,9)\}}_{d-g=4}\,\cup\,\underbrace{\{(5,8),(6,9),(7,10)\}}_{d-g=3}.$$
For $(g,d)\in\{(1,5),(2,6),(3,7)\}\cup \{(5,8),(6,9),(7,10)\}$, we are in the situation described at \eqref{GL degree bound} in Section 2, so SMRC holds, with the following exception: when $(g,d)=(7,10)$, the variety $Q^4_{10}(C)$ has expected dimension $\alpha(7,4,10)=0$, and it is in fact empty. As for $(g,d)\in\{(4,8),(5,9)\}$, we already treated them separately in the discussion around \eqref{dim of non very ample} in the same Section 2. The remaining $4$ cases for which we are able to make a complete statement and are mentioned separately in \autoref{SMRC r=4} are dealt with in the remark below. 

\begin{remark}
[the divisorial case for $Q^r_d(C)$ when $s\ge -1$]
Notice that the nonemptiness part in \autoref{deg locus in sum of pencils} implies \autoref{existence for Q^r_d,k(C)}. In particular, as long as $s=d-g-r\ge -1$ and $g\ge \binom{r}{2}-2s$, the locus $Q^r_d(C)$ is nonempty. Furthermore, when $\binom{r+2}{2}=2d+1-g$, the expected dimension of $Q^r_d(C)$ is 
$$\alpha(g,r,d)=\rho(g,r,d)-1.$$
The determinantal structure of $Q^r_d(C)$ tells us that all irreducible components of $Q^r_d(C)$ have dimension at least $\rho(g,r,d)-1$. Since we just argued that $Q^r_d(C)$ is nonempty and, as a consequence of \cite{Larson17}, we know $Q^r_d(C)\subsetneq G^r_d(C)$ is a proper subset, we deduce that in this particular instances the Strong Maximal Rank Conjecture holds i.e. $\dim Q^r_d(C)=\alpha(g,r,d)$. For $r=4$, the instances at hand are precisely $(g,d)\in \{(8,11),(6,10),(4,9),(2,8)\}$. 
\end{remark}

In what follows, we concern ourselves only with the injective area, that is, $\dim_\CC\Sym^2 V=15\le 2d+1-g=h^0\left(C,L^2\right)$. We need two preparatory lemmas. 

\begin{lemma}\label{unique and smooth quadric}
Let $C$ be a general curve of genus $g$ and let $d$ be a positive integer such that $2d+1-g\ge 15$ and $d\ge g+3$. Then for any component $\Sigma$ of $Q_d^4(C)$, the general element $\ell\in \Sigma$ has the property that the kernel of the multiplication map $\nu(\ell)$ is one dimensional and consists of a smooth quadric.
\end{lemma}
\begin{proof}
Let $\ell\in G^4_d(C)$ be such that $\dim \ker\nu(\ell)\ge 2$. If $\varphi=\varphi_\ell:C\to \PP^4$ is the induced map, then the image curve $\varphi(C)$ lies on a pencil of quadrics, so it contains at least one singular quadric. Therefore, it suffices to show that not lying on a singular quadric is a generic condition in any component of $Q^4_d(C)$. 

Since we work in $\PP^4$, singular quadrics are of rank at most $4$. From \autoref{incidence variety for rank 4 quadrics} it follows that $\dim Q^4_{d,4}(C)\le 3(d-g)-8$. On the other hand, the determinantal structure of $Q^4_{d}(C)$ implies that all its components are of dimension at least $\alpha(g,4,d)=3(d-g)-7$. This immediately proves the lemma. 
\end{proof}

\begin{lemma}\label{generic order for quadric at point}
Let $C$ be a genus $g\ge 1$ curve and let $(L,V)$ be a $g^4_d$ that maps $C$ to a curve $C_0$ in $\PP^4$. Suppose that the multiplication map 
$$\nu(L,V):\Sym^2 V\to H^0\left(C,L^2\right)$$
contains a smooth quadric $\rho$ in the kernel. Then $\ord_p(\rho)=2$ for the generic point $p$ on $C$. 
\end{lemma}
\begin{proof}
Let us translate everything in terms of the image curve $C_0$, which is known to be nondegenerate in $\PP^4$. The element $\rho\in \ker\nu(L,V)$ is a smooth quadric $Q$ in $\PP^4$ that contains $C_0$. Therefore, we trivially have that any point of the curve $C_0$ lies in $Q$ and, furthermore, any tangent vector to $C_0$ is included in the tangent space to $Q$ at the corresponding point. These mean that $Q$ has contact of order at least $2$ with $C_0$ at all points $p\in C_0$ i.e. $\ord_p(\rho)\ge 2$, for all $p\in C_0$. 

Having these said, since it is an open condition, the conclusion would follow if we ensure $\ord_p(\rho)=2$ holds for at least one point $p\in C_0$. Suppose this is not the case, so $\ord_p(\rho)\ge 3$ for all points $p\in C_0$. Geometrically, this means that all osculating $2$-planes to the curve $C_0$ are included in the tangent spaces to $Q$ at the corresponding points. 
\vspace{0.2cm}

Let $B$ be the bilinear form that defines the quadric $Q$ and let $C_0$ be given by $t\mapsto [s_0(t):s_1(t):\ldots:s_4(t)]$, where $\{s_0,s_1,\dots,s_4\}$ is a basis of global sections for $V\le H^0(C,L)$. For brevity, let us simply write $s(t)$ for the point whose coordinates are $[s_0(t):s_1(t):\ldots:s_4(t)]$. 

The normal vector to $Q$ at any point $q\in Q$ is $2B(q,-)$. Therefore, the requirement that the osculating $2$-planes to $C_0$ are included in the tangent spaces to the quadric is just 
\begin{align}\label{osculating 2-plane translate}
    B(s(t),s''(t))=0, \textnormal{ for all }t\in C.
\end{align}
Now, since $C_0$ lies in $Q$, we also have $B(s(t),s(t))=0$. Differentiating in $t$, we obtain the equality $B(s(t),s'(t))=0$, which means precisely that the tangent direction at $s(t)\in C_0$ lies in $T_{s(t)}Q$, already known. Differentiating again in $t$ gives the relation $B(s'(t),s'(t))+B(s(t),s''(t))=0$. Using \eqref{osculating 2-plane translate}, this is just $B(s'(t),s'(t))=0$. But then $B(s(t)+\lambda s'(t),s(t)+\lambda s'(t))=0$ for all $\lambda\in \CC$, so the entire tangent line to $s(t)$ at $C_0$ must lie inside the quadric $Q$, for all points $s(t)$ on $C_0$. But the Fano variety of lines in a (smooth) quadric threefold is the orthogonal Grassmannian $OG(2,5)\cong \PP^3$, so the tangent directions to the curve $C_0$ have only three degrees of freedom, contradicting the nondegeneracy of $C_0$. 
\end{proof}

We continue the discussion when $r=4$ when $\binom{r+2}{2}=15<2d+1-g$. We present an inductive argument that guarantees the existence of a component of $Q^4_d(C)$ that has the expected dimension
$$\alpha(g,4,d)=3s+5.$$
Observe that, for fixed values of $d-g$ (or fixed values of $s$), the expected dimension is the same; see \autoref{r=4} for an explicit view of the evolution of $\alpha(g,4,d)$. The induction goes through each column where $d-g$ has a fixed value, and given the existence of component of the expected dimension $\alpha(g-1,4,d-1)=3s+5$ in $Q^4_{d-1}(D)$ for the general curve $[D]\in\mrond_{g-1}$, we show there exists a component of the expected dimension $\alpha(g,4,d)=3s+5$ in $Q^4_d(C)$ for the general curve $[C]\in\mrond_g$. 

As base cases, we use $(g,d)\in\{(8,11),(6,10)\}$ for the area $-s=g-d+r\in\{0,1\}$, as well as all pairs $(g,d)$ with $\alpha(g,4,d)\ge g$ for the area $-s\le -1$ where, for a general curve $[C]\in\mrond_g$, we infer from \autoref{SMRC r ge 5} the existence of a component in $Q^4_d(C)$ of the expected dimension $\alpha(g,4,d)=3s+5$, which also dominates the Jacobian $\Pic^d(C)$; in \autoref{r=5}, this is the "upper-right triangle" range that is delimited for both $r=4$ and $r=5$. 

So let $(g,d)$ be such that $\binom{r+2}{2}=15<2d+1-g$ and fix $s=d-g-4\ge -1$. Consider a curve $C_0=D_0\cup_p E_0$ in $\ol{\mrond}_{g}$, where $[D_0,p]\in\mrond_{g-1,1}$ and $[E_0,p]\in\mrond_{1,1}$. Let $M_g$ be the versal deformation space of $C_0$, so $M_g$ may be viewed as an étale neighbourhood of $[C_0]$ in $\ol{\mrond}_g$. The technical ingredient in the induction process is the following stack.

\begin{theorem}\label{stack parametrize mrc}
There is a stack $\mu:\mathfrak{Q}^4_d\to M_g$ with the following structure:

$\bullet$ if $[C]\in M_g$ is a smooth curve, then $\mu^{-1}[C]$ parametrizes tuples $(\ell,\rho)$, where $\ell$ is a linear system in $G^4_d(C)$ and $\rho\in \PP\ker\nu(\ell)$, 

$\bullet$ if $[C=D\cup_p E]\in M_g$ is a singular curve such that $[D,p]\in \mrond_{g-1,1}$ and $[E,p]\in\mrond_{1,1}$, then $\mu^{-1}[C]$ parametrizes tuples $(\ell,\rho_D,\rho_E)$ such that 

$\diamondsuit$ $\ell=\left\{(L_D,V_D),(L_E,V_E)\right\}$ is a limit $g^4_d$ on $C$, 

$\diamondsuit$ $\rho_D\in \PP\ker\left\{\Sym^2 V_D\to H^0\left(D,L_D^2\right)\right\}$, 

$\diamondsuit$ $\rho_E\in \PP\ker\left\{\Sym^2 V_E\to H^0\left(E,L_E^2\right)\right\}$, and 

$\diamondsuit$ the compatibility condition $\ord_p(\rho_D)+\ord_p(\rho_E)\ge 2d$ holds. 
\end{theorem}
\begin{proof}
From \cite{Osserman06}, there is a stack $\sigma:\grond^4_d\to M_g$ that parametrizes (limit) linear series over each curve $[C]\in M_g$. Let $\crond_g\to M_g$ be the universal curve and let 
$$\pi:\widehat{\crond}_g:=\crond_g\times_{M_g}\grond^r_d\to \grond^4_d$$
be the induced universal curve. We fix a Poincaré line bundle $\lrond$ over $\widehat{\crond}_g$ such that for a curve $C=D\cup_p E$ with $[D,p]\in \mrond_{g-1,1}$ and $[E,p]\in\mrond_{1,1}$ and a limit series $\ell=\left\{(L_D,V_D),(L_E,V_E)\right\}$ on $C$, we have $\lrond\vert_D=L_D$ and $\lrond\vert_E=L_E(-dp)$. As in \cite{FJP25} Section 3, we use $\lrond$ and construct vector bundles $\erond$ and $\frond$ over $\grond^4_d$ such that

$\bullet$ for $t=[C,\ell=(L,V)]\in \grond^4_d$ with $C$ smooth, $\erond(t)=V$ and $\frond(t)=H^0\left(C,L^2\right)$, and 

$\bullet$ for $t=[C=D\cup_p E,\ell]\in \grond^4_d$ with $[D,p]\in\mrond_{g-1,1}$, $[E,p]\in \mrond_{1,1}$ and a limit series $\ell=\left\{(L_D,V_D),(L_E,V_E)\right\}$, we have $\erond(t)=V_D$ and $\frond(t)=H^0\left(C,L_D^2\right)$.
\\
We also have a vector bundle morphism $\nu_D:\Sym^2\erond\to \frond$ that globalizes the multiplication map of sections. Let us denote with $\PP\ker(\nu_D)\subseteq \PP(\Sym^2 \erond)$ the indeterminacy locus of the induced rational map $\PP(\nu_D):\PP(\Sym^2 \erond)\dashrightarrow \PP(\frond)$. If $\Delta\subset M_g$ is the locus parametrizing the curves for which the node $p$ is not a smooth point, we consider $\erond_p$ to be the closure in $\widehat{\crond}_g$ of the components of $\sigma^{-1}(\Delta)$ containing $E\setminus\{p\}$. Then $\lrond(d\erond_p)$ is also a Poincaré line bundle on $\widehat{\crond}_g$ for which, by contrast to $\lrond$, over a nodal curve $t=[C=D\cup_p E,\ell]$ as above, the induced vector bundles $\erond$ and $\frond$ have local structure $\erond(t)=V_E$ and $\frond(t)=H^0\left(C,L_E^2\right)$. As above, we consider the indeterminacy locus $\PP\ker(\nu_E)\subseteq \PP(\Sym^2\erond)$ of the induced rational map $\PP(\nu_E):\PP(\Sym^2 \erond)\dashrightarrow\PP(\frond)$. 

Then we define $\mu:\mathfrak{Q}^4_d:=\PP\ker(\nu_D)\times_{M_g}\PP\ker(\nu_E)\to M_g$ to be the stack that we desire. Indeed, for a point $[C=D\cup_pE]$ as above, the compatibility condition comes from passing to limit in a family whose general fiber is a smooth curve and the special fiber is a nodal curve as above, and it is an immediate consequence of e.g. Lemma 4.1\cite{Farkas05}. 
\end{proof}

The stack $\mathfrak{Q}^4_d$ has a determinantal structure over $M_g$ and all irreducible components of each fiber $\mu^{-1}[C]$ have dimension at least $\alpha(g,4,d)=3s+5$. 

Let $C=D\cup_p E$ be as above, such that $[D,p]\in\mrond_{g-1,1}$ is general. We show that $\mu^{-1}[C]$ contains at least one component of the expected dimension $3s+5$. Then, by semicontinuity, the same is true for the general curve $[C']\in M_g$. Forgetting the attached quadric $\rho\in \PP\ker\nu(\ell)$ which is also quantified by $\mu^{-1}[C']$, \autoref{unique and smooth quadric} tells us that $Q^4_d(C')$ itself contains a component of the expected dimension $3s+5$ for the general curve $[C']\in M_g$, as required. 

Now, from the induction hypothesis, there exists a component $\Sigma$ of the expected dimension $3s+5$ in $Q^4_{d-1}(D)$. Thanks to \autoref{unique and smooth quadric} and \autoref{generic order for quadric at point}, $p\in D$ can be chosen general enough such that the following hold for the general $\ell=(L,V)\in \Sigma$: 

$\bullet$ $V$ has ramification sequence $(0,1,2,3,4)$ at $p$, 

$\bullet$ there is a unique quadic $\rho\in \PP\ker\left\{\Sym^2 V\to H^0\left(D,L^2\right)\right\}$, and 

$\bullet$ $\ord_p(\rho)=2$. 
\vspace{0.2cm}

Pick a general $\ell=(L,V)\in \Sigma$ and set $L_D=L(p)$, $V_D=p+V$. Forgetting the basepoint $p$, any $\rho_D\in \PP\ker\left\{\Sym^2 V_D\to H^0\left(D,L_D^2\right)\right\}$ identifies with the element $\rho$ described above. Therefore, $\rho_D$ is unique and $\ord_p(\rho_D)=2+2=4$.

Take a linear series $(L_E,V_E)\in G^4_d(E)$ that is compatible with $(L_D,V_D)$, as in the requirements of \autoref{stack parametrize mrc}. 
If $L_E=\OO_E((d-1)p+q)$ for some point $q\ne p$, then $V_E$ needs to have ramification sequence at $p$ at least $(d-5,d-4,\dots,d-1)$. Therefore, we can write $V_E=(d-5)p+W$, where $W$ is a $5$-dimensional subspace of $H^0(E,\OO_E(4p+q))$. But $H^0(E,\OO_E(4p+q))$ is itself $5$-dimensional, so we need $V_E=(d-5)p+H^0(E,\OO_E(4p+q))$. Now, the compatibility condition requires $\ord_p(\rho_E)\ge 2d-4$, which means that both sections in each tensor of $\rho_E$ have order at least $d-3$ at $p$. Consequently, removal of the basepoint $p$ identifies $\rho_E$ with an element of the kernel of 
$$\Sym^2H^0(E,\OO_E(2p+q))\to H^0(E,\OO_E(4p+2q)).$$
But this multiplication map is an isomorphism, so no such $\rho_E$ could exist! Thus, one needs $L_E=\OO_E(E,\OO_E(dp))$. Then $V_E$ is again a $5$-dimensional subspace of $H^0(E,\OO_E(dp))$ and has ramification sequence at least $(d-5,d-4,\dots,d-1)$. We conclude, as in the previous paragraph, that $V_E=(d-5)p+H^0(E,\OO_E(5p))$. Again, the compatibility condition requires $\ord_p(\rho_E)\ge 2d-4$, so the sections in each tensor of $\rho_E$ each have order at least $d-4$ at $p$. Removal of the basepoint $p$ now identifies $\rho_E$ with an element of the kernel of
$$\Sym^2H^0(E,\OO_E(4p))\to H^0(E,\OO_E(8p)),$$
which also satisfies $\ord_p(\rho_E)\ge 4$. 

\begin{claim}
There is a unique such $\rho_E$.
\end{claim}
\begin{proof}
The multiplication map $\Sym^2H^0(E,\OO_E(4p))\to H^0(E,\OO_E(8p))$ is surjective as of \eqref{GL degree bound} in Section 2. If we let $W$ be the vector subspace of $\Sym^2 H^0(E,\OO_E(4p))$ that is spanned by all tensors $\sigma\otimes \tau$ for which $\ord_p(\sigma)+\ord_p(\tau)\ge 4$, then surjectivity ensures that the image of $W$ is precisely the subspace $4p+H^0(E,\OO_E(4p))\le H^0(E,\OO_E(8p))$, and this is $4$-dimensional.

At the same time, since the vanishing sequence of $H^0(E,\OO_E(4p))$ at $p$ is $(0,1,2,4)$, there are only $5$ ways to pair vanishing orders in order to get a sum of at least $4$, and they are $4+0,4+1,4+2,4+4$ and $2+2$. Since any element $\rho_E$ in the kernel of the original multiplication map of order $\ge 4$ must be in the kernel of the restriction to $W$, we conclude that $\rho_E$ varies in a $1$-dimensional vector space. Therefore, up to projectivization, $\rho_E$ is unique. 
\end{proof}

All in all, as we sweep with $(L_D,V_D)=(L(p),p+V)$ through a nonempty open of the component $\Sigma$ of $Q^4_{d-1}(D)$, the elements $\rho_D$, $(L_E,V_E)$ and $\rho_E$ are uniquely determined. It follows that $\mu^{-1}[C]$ has a component of dimension $3s+5$. This finishes the proof of \autoref{SMRC r=4}. 

\section{Annex with Tables for SMRC in $\PP^3$, $\PP^4$ and $\PP^5$}

The following $3$ tables show the evolution of the expected dimension $\alpha(g,r,d)$, when $r\in\{3,4,5\}$ and the values of the difference $d-g$ are taken one at a time. The colors stand for the following:

$\bullet$ by a \textbf{\textcolor{Blue}{blue}} cell, we mean that the SMRC holds,

$\bullet$ by a \textbf{\textcolor{Green}{green}} cell, we mean that there exists a component of the expected dimension $\alpha(g,r,d)$,

$\bullet$ by a \textbf{\textcolor{LightGreen}{light green}} cell, we mean that the existence part of $Q^r_d(C)$ is established (see in the table for $r=5$), 

$\bullet$ by a \textbf{\textcolor{Red}{red}} cell, we mean that SMRC predicts a wrong dimensionality statement; in this case we also write, in parenthesis, what is the actual dimension of $Q^r_d(C)$.

\begin{figure}[h]
\centering
\includegraphics[width=6.2in]{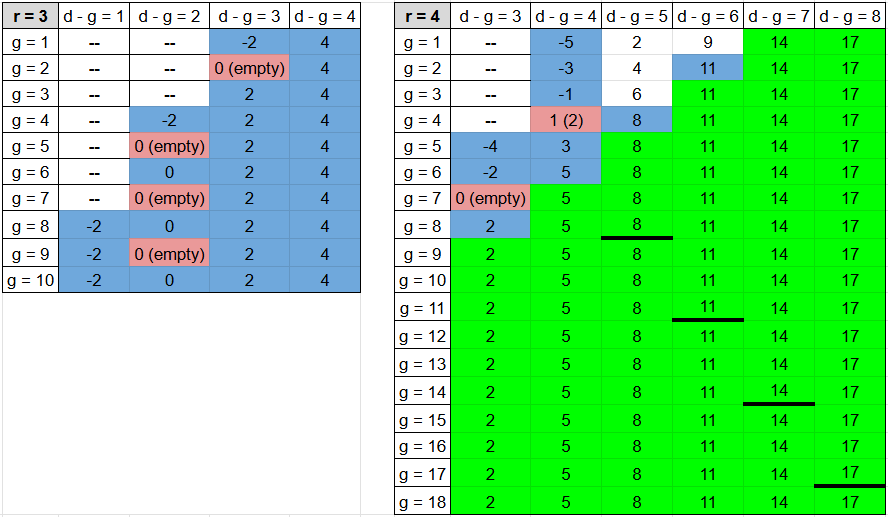}
\label{r=3}
\label{r=4}
\end{figure}

\begin{figure}[h]
\centering
\includegraphics[width=4.3in]{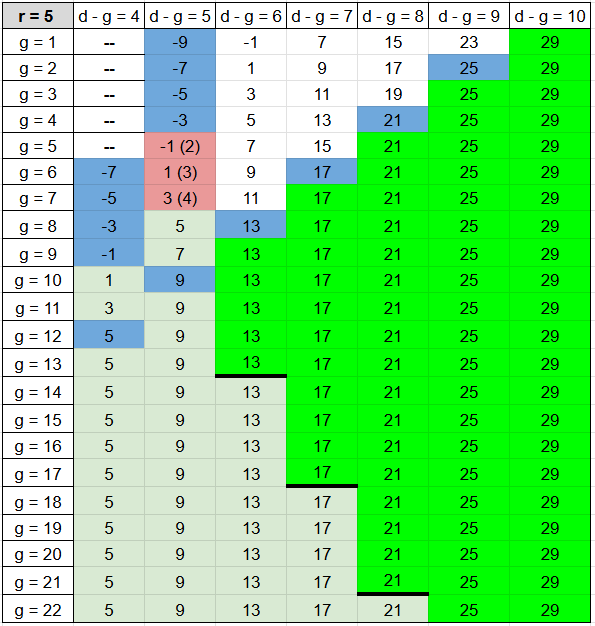}
\label{r=5}
\end{figure}


\end{document}